\documentclass[10pt]{amsart}
\usepackage{amsfonts}
\usepackage{mathrsfs}
\usepackage{amscd}
\usepackage{amsmath}
\usepackage{amssymb}
\usepackage{latexsym}
\usepackage{lscape}
\usepackage{xypic}
\usepackage{comment}
\usepackage{amscd}
\usepackage{wasysym}  
\usepackage{tikz} 
\usetikzlibrary{matrix,arrows}
\usepackage{appendix}
\usepackage{geometry}
\usepackage[nice]{nicefrac}
\usepackage{dsfont}

\usepackage[pagebackref,hyperindex,linktocpage=true]{hyperref}
\hypersetup{
    colorlinks,
    linkcolor={red!50!black},
    citecolor={blue!50!black},
    urlcolor={blue!80!black}
}

\usepackage{color}

\newcommand{\cb}{\color{blue}}
\newcommand{\cm}{\color{magenta}}

\newtheorem{thm}{Theorem}[section]

\newtheorem{prop}[thm]{Proposition}
\newtheorem{lem}[thm]{Lemma}

\newtheorem{lem-def}[thm]{Lemma-Definition}
\newtheorem{cor}[thm]{Corollary}
\newtheorem{conject}[thm]{Conjecture}

\theoremstyle{definition}

\newtheorem{rmk}[thm]{Remark}
\newtheorem{dfn}[thm]{Definition}

\numberwithin{equation}{section}

\newcommand{\nc}{\newcommand}

\nc{\on}{\operatorname}
\nc{\fraka}{{\mathfrak a}} \nc{\bba}{{\mathbf a}}
\nc{\frakb}{{\mathfrak b}}
\nc{\frakc}{{\mathfrak c}}
\nc{\frakd}{{\mathfrak d}}
\nc{\frake}{{\mathfrak e}}
\nc{\frakf}{{\mathfrak f}}
\nc{\frakg}{{\mathfrak g}}
\nc{\frakh}{{\mathfrak h}}
\nc{\fraki}{{\mathfrak i}}
\nc{\frakj}{{\mathfrak j}}
\nc{\frakk}{{\mathfrak k}}
\nc{\frakl}{{\mathfrak l}}
\nc{\frakm}{{\mathfrak m}}
\nc{\frakn}{{\mathfrak n}}
\nc{\frako}{{\mathfrak o}}
\nc{\frakp}{{\mathfrak p}}
\nc{\frakq}{{\mathfrak q}}
\nc{\frakr}{{\mathfrak r}}
\nc{\fraks}{{\mathfrak s}}
\nc{\frakt}{{\mathfrak t}}
\nc{\fraku}{{\mathfrak u}}
\nc{\frakv}{{\mathfrak v}}
\nc{\frakw}{{\mathfrak w}}
\nc{\frakx}{{\mathfrak x}}
\nc{\fraky}{{\mathfrak y}}
\nc{\frakz}{{\mathfrak z}}
\nc{\frakA}{{\mathfrak A}}
\nc{\frakB}{{\mathfrak B}}
\nc{\frakC}{{\mathfrak C}}
\nc{\frakD}{{\mathfrak D}}
\nc{\frakE}{{\mathfrak E}}
\nc{\frakF}{{\mathfrak F}}
\nc{\frakG}{{\mathfrak G}}
\nc{\frakH}{{\mathfrak H}}
\nc{\frakI}{{\mathfrak I}}
\nc{\frakJ}{{\mathfrak J}}
\nc{\frakK}{{\mathfrak K}}
\nc{\frakL}{{\mathfrak L}}
\nc{\frakM}{{\mathfrak M}}
\nc{\frakN}{{\mathfrak N}}
\nc{\frakO}{{\mathfrak O}}
\nc{\frakP}{{\mathfrak P}}
\nc{\frakQ}{{\mathfrak Q}}
\nc{\frakR}{{\mathfrak R}}
\nc{\frakS}{{\mathfrak S}}
\nc{\frakT}{{\mathfrak T}}
\nc{\frakU}{{\mathfrak U}}
\nc{\frakV}{{\mathfrak V}}
\nc{\frakW}{{\mathfrak W}}
\nc{\frakX}{{\mathfrak X}}
\nc{\frakY}{{\mathfrak Y}}
\nc{\frakZ}{{\mathfrak Z}}
\nc{\bbA}{{\mathbb A}}
\nc{\bbB}{{\mathbb B}}
\nc{\bbC}{{\mathbb C}}
\nc{\bbD}{{\mathbb D}}
\nc{\bbE}{{\mathbb E}}
\nc{\bbF}{{\mathbb F}} \nc{\bbf}{{\mathbf f}}
\nc{\bbG}{{\mathbb G}}
\nc{\bbH}{{\mathbb H}}
\nc{\bbI}{{\mathbb I}}
\nc{\bbJ}{{\mathbb J}}
\nc{\bbK}{{\mathbb K}}
\nc{\bbL}{{\mathbb L}}
\nc{\bbM}{{\mathbb M}}
\nc{\bbN}{{\mathbb N}}
\nc{\bbO}{{\mathbb O}}
\nc{\bbP}{{\mathbb P}}
\nc{\bbQ}{{\mathbb Q}}
\nc{\bbR}{{\mathbb R}}
\nc{\bbS}{{\mathbb S}}
\nc{\bbT}{{\mathbb T}}
\nc{\bbU}{{\mathbb U}}
\nc{\bbV}{{\mathbb V}}
\nc{\bbW}{{\mathbb W}}
\nc{\bbX}{{\mathbb X}}
\nc{\bbY}{{\mathbb Y}}
\nc{\bbZ}{{\mathbb Z}}
\nc{\calA}{{\mathcal A}}
\nc{\calB}{{\mathcal B}}
\nc{\calC}{{\mathcal C}}
\nc{\calD}{{\mathcal D}}
\nc{\calE}{{\mathcal E}}
\nc{\calF}{{\mathcal F}}
\nc{\calG}{{\mathcal G}}
\nc{\calH}{{\mathcal H}}
\nc{\calI}{{\mathcal I}}
\nc{\calJ}{{\mathcal J}}
\nc{\calK}{{\mathcal K}}
\nc{\calL}{{\mathcal L}}
\nc{\calM}{{\mathcal M}}
\nc{\calN}{{\mathcal N}}
\nc{\calO}{{\mathcal O}}
\nc{\calP}{{\mathcal P}}
\nc{\calQ}{{\mathcal Q}}
\nc{\calR}{{\mathcal R}}
\nc{\calS}{{\mathcal S}}
\nc{\calT}{{\mathcal T}}
\nc{\calU}{{\mathcal U}}
\nc{\calV}{{\mathcal V}}
\nc{\calW}{{\mathcal W}}
\nc{\calX}{{\mathcal X}}
\nc{\calY}{{\mathcal Y}}
\nc{\calZ}{{\mathcal Z}}

\nc{\scrA}{{\mathscr A}}
\nc{\scrB}{{\mathscr B}}
\nc{\scrR}{{\mathscr R}}

\nc{\Bmu}{\mbox{$\raisebox{-0.59ex}{$l$}\hspace{-0.18em}\mu\hspace{-0.88em}\raisebox{-0.98ex}{\scalebox{2}{$\color{white}.$}}\hspace{-0.416em}\raisebox{+0.88ex}{$\color{white}.$}\hspace{0.46em}$}{}}

\nc{\bnu}{{\bar{ \nu}}}

\nc{\olO}{\bar{\calO}}

\nc{\al}{{\alpha}} 
\nc{\be}{{\beta}}
\nc{\ga}{{\gamma}} \nc{\Ga}{{\Gamma}}
 \nc{\hGa}{\hat{\Gamma}}
\nc{\ve}{{\varepsilon}} 
\nc{\la}{{\lambda}} \nc{\La}{{\Lambda}}
\nc{\om}{\omega} \nc{\Om}{\Omega} 
\nc{\sig}{{\sigma}} \nc{\Sig}{{\Sigma}}

\nc{\tnb}{\psi_{\rm tame}}
\nc{\oM}{\overline{{M}}}
\nc{\op}{{\on{op}}}
\nc{\ad}{{\on{ad}}}
\nc{\alg}{{\on{alg}}}
\nc{\Ad}{{\on{Ad}}}
\nc{\Adm}{{\on{Adm}}} \nc{\aff}{{\on{aff}}}
\nc{\Aut}{{\on{Aut}}}
\nc{\Bun}{{\on{Bun}}}
\nc{\cha}{{\on{char}}}
\nc{\der}{{\on{der}}}
\nc{\Der}{{\on{Der}}}
\nc{\diag}{{\on{diag}}}
\nc{\End}{{\on{End}}}
\nc{\Fl}{{\calF\!\ell}}
\nc{\Tr}{{\on{Transp}}}
\nc{\TR}{{\calT\!\calR}}
\nc{\Gal}{{\on{Gal}}}
\nc{\Gr}{{\on{Gr}}}
\nc{\rH}{{\on{H}}}
\nc{\Hom}{{\on{Hom}}}
\nc{\IC}{{\on{IC}}}
\nc{\id}{{\on{id}}}
\nc{\Id}{{\on{Id}}}
\nc{\ind}{{\on{ind}}}
\nc{\Ind}{{\on{Ind}}}
\nc{\Lie}{{\on{Lie}}}
\nc{\Pic}{{\on{Pic}}}
\nc{\pr}{{\on{pr}}}
\nc{\Res}{{\on{Res}}}
\nc{\res}{{\on{res}}} \nc{\Sat}{{\on{Sat}}}
\nc{\s}{{\on{sc}}}
\nc{\drv}{{\on{der}}}
\nc{\sgn}{{\on{sgn}}}
\nc{\Spec}{{\on{Spec}}}\nc{\Spf}{\on{Spf}} 
\nc{\Sph}{\on{Sph}}
\nc{\St}{{\on{St}}}
\nc{\tr}{{\on{tr}}}
\nc{\Mod}{{\mathrm{-Mod}}}
\nc{\Hilb}{{\on{Hilb}}} 
\nc{\Ext}{{\on{Ext}}} 
\nc{\vs}{{\on{Vec}}}
\nc{\ev}{{\on{ev}}}
\nc{\nO}{{\breve{\calO}}}
\nc{\tS}{{\tilde{S}}}
\nc{\spe}{{\on{sp}}}
\nc{\loc}{{\on{loc}}}

\nc{\nscrR}{{\mathscr{R}^{\on{nr}}}}

\nc{\GL}{{\on{GL}}}
\nc{\U}{{\on{U}}}
\nc{\Gl}{\on{Gl}} 
\nc{\GSp}{{\on{GSp}}}
\nc{\gl}{{\frakg\frakl}}
\nc{\SL}{{\on{SL}}} 
\nc{\SU}{{\on{SU}}} 
\nc{\SO}{{\on{SO}}}
\nc{\PGL}{{\on{PGL}}}

\nc{\Conv}{{\on{Conv}}}
\nc{\Rep}{{\on{Rep}}}
\nc{\Dom}{{\on{Dom}}}
\nc{\red}{{\on{red}}}
\nc{\act}{{\on{act}}}
\nc{\nr}{{\on{nr}}}
\nc{\ctf}{{\on{ctf}}}

\nc{\str}{{\on{-}}} 
\nc{\os}{{\bar{s}}}
\nc{\oeta}{{\bar{\eta}}}

\nc{\hookto}{\hookrightarrow}
\nc{\longto}{\longrightarrow}
\nc{\leftto}{\leftarrow}
\nc{\onto}{\twoheadrightarrow}
\nc{\lonto}{\twoheadleftarrow}

\nc{\uG}{{\underline{G}}}
\nc{\uA}{{\underline{A}}}
\nc{\uS}{{\underline{S}}}
\nc{\uT}{{\underline{T}}}
\nc{\uM}{{\underline{M}}}
\nc{\uP}{{\underline{P}}}
\nc{\uB}{{\underline{B}}}
\nc{\uN}{{\underline{N}}}

\nc{\ucG}{{\underline{\calG}}}
\nc{\ucA}{{\underline{\calA}}}
\nc{\ucS}{{\underline{\calS}}}
\nc{\ucT}{{\underline{\calT}}}
\nc{\ucM}{{\underline{\calM}}}
\nc{\ucP}{{\underline{\calP}}}
\nc{\ucN}{{\underline{\calN}}}

\nc{\bF}{{\breve{F}}}

\nc{\oFl}{{\overline{\Fl}}} 
\nc{\bU}{{\overline{U}}}
\nc{\tGr}{{\tilde{\Gr}}}
\nc{\cGr}{\calG\! r}
\nc{\oGr}{\overline{\on{Gr}}} 
\nc{\ocGr}{\overline{\calG\! r}}
\nc{\co}{{\colon}}
\nc{\sch}[1]{(Sch/{#1})}
\nc{\HypLoc}[1]{HypLoc({#1})}

\nc{\ohtimes}{\stackrel{!}{\otimes}}
\nc{\boxtilde}{\widetilde{\boxtimes}}
\nc{\vstar}{{\varhexstar}}

\nc{\Div}{\on{Div}}

\nc{\bslash}{\backslash}
\nc{\algQl}{{\bar{\bbQ}_\ell}}
\nc{\sF}{{\bar{F}}}
\nc{\nF}{{\breve{F}}}
\nc{\nW}{{W^{\on{nr}}}}
\nc{\sk}{{\bar{k}}}
\nc{\cont}{\on{c}}
\nc{\Supp}{\on{Supp}}
\nc{\blt}{\bullet}  
\nc{\dom}{\on{dom}}
\nc{\scon}{{\on{sc}}} 
\nc{\Affine}{\on{Aff}} 
\nc{\nscrA}{\mathscr{A}^{\on{nr}}} 
\nc{\nfraka}{{\bbf^{\on{nr}}}}
\nc{\ran}{{\rangle}}
\nc{\lan}{{\langle}}
\nc{\bk}{{\bar{k}}}
\nc{\tF}{{\tilde{F}}}
\nc{\sS}{{\bar{S}}}
\nc{\LG}{{^\text{L}\hspace{-0.04cm}G}}
\nc{\LL}{{^\text{L}\hspace{-0.07cm}L}}

\nc{\pot}[1]{ [\hspace{-0,5mm}[ {#1} ]\hspace{-0,5mm}] }
\nc{\rpot}[1]{ (\hspace{-0,7mm}( {#1} )\hspace{-0,7mm}) }

\nc{\defined}{\hspace{0.1cm}\stackrel{\text{\tiny \rm def}}{=}\hspace{0.1cm}}

\begin{document}

\title[Smoothness of Schubert varieties]{Smoothness of Schubert varieties in \\ twisted affine Grassmannians}
\author[T.\,J.\,Haines and T.\,Richarz]{by Thomas J. Haines and Timo Richarz}

\address{Department of Mathematics, University of Maryland, College Park, MD 20742-4015, DC, USA}
\email{tjh@math.umd.edu}

\address{Technical University of Darmstadt, Department of Mathematics, 64289 Darmstadt, Germany}
\email{richarz@mathematik.tu-darmstadt.de}

\thanks{Research of T.H.~partially supported by NSF DMS-1406787, and research of T.R.~ funded by the Deutsche Forschungsgemeinschaft (DFG, German Research Foundation) - 394587809.}

\maketitle

\begin{center}
\textit{To Michael Rapoport on his 70th birthday}
\end{center}

\begin{abstract} We give a complete list of smooth and rationally smooth normalized Schubert varieties in the twisted affine Grassmannian associated with a tamely ramified group and a special vertex of its Bruhat-Tits building. The particular case of the quasi-minuscule Schubert variety in the quasi-split but non-split form of $\on{Spin}_8$ (``ramified triality'') provides an input needed in the article by He-Pappas-Rapoport classifying Shimura varieties with good or semi-stable reduction. 
\end{abstract}

\tableofcontents


\thispagestyle{empty}

\section{Introduction}

Let $k$ be an algebraically closed field, and let $G$ be a connected reductive group over the Laurent series field $F=k\rpot{t}$. Associated with any special vertex $x$ of the Bruhat-Tits building is the twisted affine Grassmannian $\Gr_{G,x}$. Under the additional assumption that $G$ splits over a tamely ramified extension of $F$, we give a complete answer to the question of whether a given (normalized) Schubert variety in $\Gr_{G,x}$ is smooth or singular (resp.~rationally smooth or not rationally smooth). 

If $G$ is split and ${\rm char}(k) = 0$, then such a classification is known by the work of Evens-Mirkovi\'c \cite{EM99} and Malkin-Ostrik-Vybornov \cite{MOV05}. The answer is strikingly simple: the Schubert variety ${\rm Gr}^{\leq \mu}_{G,x}$ corresponding to a cocharacter $\mu \in X_*(G)$ is smooth if and only if $\mu$ is minuscule. 

If $G$ is not split, then our classification has a similar flavor, but the phenomenon of {\it exotic smoothness} enters in: there are surprising additional cases of smoothness, where the group $G$ is a ramified odd unitary group and $\mu$ is quasi-minuscule. Unlike the split case, the nature of the special vertex $x$ now plays a pivotal role which was first observed by the second named author in \cite[Prop 4.16]{Arz09}; see Theorem \ref{sm_thm} for a precise statement. 

Our work is intertwined with the work of He-Pappas-Rapoport \cite{HPR} which classifies Shimura varieties with good or semi-stable reductions by giving a corresponding classification of (slight modifications of) Pappas-Zhu local models \cite{PZ13}. 
The connection between this article and \cite{HPR} arises in the following way: in \cite[Thm.\,9.1]{PZ13}, it is proved that the special fiber of any local model $\mathbb M^{\rm loc}_K(G,\{\mu\})$ is isomorphic to an explicit union of Schubert varieties in a (twisted) partial affine flag variety over $k$. The smooth local models are those whose special fiber is a single smooth Schubert variety. By \cite[Thm.\,1.2]{HPR}, this implies that the parahoric $K=K_x$ is a special maximal parahoric associated with some special vertex $x$, and occurs in the following situations: either $G$ is split so that $x$ is hyperspecial, or $G$ is non-split and the triple $(G,\mu,x)$ is of {\em exotic good reduction type}. Exotic good reduction comes in three kinds: 
\begin{itemize}
\item[1)] even unitary exotic, discovered by Pappas-Rapoport \cite[5.3]{PR09}, 
\item[2)] odd unitary exotic, discovered by the second named author \cite[Prop.\,4.16]{Arz09}, and 
\item[3)] orthogonal exotic, discovered by He-Pappas-Rapoport \cite[$\S5.1$]{HPR}. 
\end{itemize}
In cases 1) and 3), the corresponding Schubert variety $\Gr_{G,x}^{\leq \bar{\mu}}$ is minuscule (hence smooth), and the choice of special vertex $x$ plays no role. In case 2), the Schubert variety $\Gr_{G,x}^{\leq \bar{\mu}}$ is quasi-minuscule, and the choice of special vertex plays a crucial role. This case relates to the phenomenon of exotic smoothness in twisted affine Grassmannians. 

\medskip

\subsection{Statement of the results} Let $k$ be an algebraically closed field, and let $F=k\rpot{t}$ be the formal Laurent series field, with absolute Galois group $I$. Let $G$ be a connected reductive group over $F$ which is adjoint, absolutely simple, and splits over a tamely ramified extension of $F$. Associated to every special vertex $x$ in the Bruhat-Tits building, we have the twisted affine Grassmannian $\Gr_{G,x}$. If $G$ is split, all special vertices are conjugate under $G_{\on{ad}}(F)$. If $G$ is not split, this is no longer true (cf.\,\cite[$\S2.5$]{Ti77}). This fact plays an important role in the phenomenon of exotic smoothness of Schubert varieties.

We choose further a pair $T\subset B\subset G$ of a maximal torus and a Borel subgroup defined over $F$ which are in good position with respect to $x$, cf.\,$\S\ref{Ratl_Smooth_Sec}$ below. Associated with each dominant $\bar{\mu}\in X_*(T)^+_I$ is the Schubert variety $\Gr_{G,x}^{\leq \bar{\mu}}\subset \Gr_{G,x}$ which is an irreducible projective $k$-variety. 

Let $M$ denote the set of minimal elements of $X_*(T)_I^+\backslash \{0\}$ with respect to the partial ordering $\leq$ defined by the \'{e}chelonnage coroots $\breve{\Sigma}^\vee \subset X_*(T)_I$, cf.~\cite{Hai18}. Recall that $\bar{\mu} \in M$ is 
\begin{itemize}
\item {\em minuscule} if $\langle \alpha, \bar{\mu} \rangle \in \{0, \pm 1\}$ for all roots $\alpha \in \breve{\Sigma}$
\item {\em quasi-minuscule}, otherwise.
\end{itemize}
In the second case, there exists a unique root $\gamma \in \breve{\Sigma}$ with $\langle \gamma, \bar{\mu} \rangle \geq 2$ and $\gamma$ is necessarily a highest root, and $\bar{\mu} = \gamma^\vee$. Further, $\langle \alpha, \bar{\mu} \rangle \in \{0, \pm 1, \pm 2\}$ for all $\alpha \in \breve{\Sigma}$, cf.~\cite[Lem.\,1.1]{NP01}. Conversely, if $\bar{\mu} \in X_*(T)_I^+\backslash \{0\}$ belongs to the coroot lattice, and if $|\langle \alpha, \bar{\mu}\rangle| \leq 2, \,\, \forall \alpha \in \breve{\Sigma}$, then $\bar{\mu} \in M$ and hence $\bar{\mu}$ is quasi-minuscule. Therefore, any irreducible root system possesses a {\em unique} quasi-minuscule coweight.

Geometrically, $\bar{\mu}$ being minuscule means that $\Gr_{G,x}^{\leq \bar{\mu}}=\Gr_{G,x}^{\bar{\mu}}$ is a single stratum, whereas $\bar{\mu}$ being quasi-minuscule means that $\Gr_{G,x}^{\leq \bar{\mu}}=\Gr_{G,x}^{\bar{\mu}}\amalg\{e\}$ where $e\in \Gr_{G,x}(k)$ is the base point.

Under the identification $X_*(T)_I = X^*((T^\vee)^I)$, the \'{e}chelonnage coroots $\breve{\Sigma}^\vee$ correspond to the roots for $((G^\vee)^I, (T^\vee)^I)$ by \cite[$\S5.1$]{Hai18} where $(G^\vee)^I$ is a simple and semi-simple connected reductive group with maximal torus $(T^\vee)^I$, cf.~Proposition \ref{fix_pt_prop}. Note that $\bar{\mu} \in X_*(T)_I$ is (quasi-)minuscule with respect to $\breve{\Sigma}$ if and only if it is (quasi-)minuscule when viewed as a $(T^\vee)^I$-weight. Similarly, a fundamental $(T^\vee)^I$-weight $\omega_i$ can be viewed as an element $\omega_i \in X_*(T)_I$. Our main results are as follows. 

\begin{thm}\label{ratl_sm_thm} Let $\bar{\mu}\in X_*(T)_I\bslash \{0\}$ be dominant. The Schubert variety $\Gr_{G,x}^{\leq \bar{\mu}}$ is rationally smooth if and only if $x\in \scrB(G,F)$ is any special vertex and the pair $(G,\bar{\mu})$ belongs up to isomorphism to the following list:
\begin{itemize}
\item any $G$, and $\bar{\mu}$ minuscule \textup{(}for a complete list, see \cite[\S5.2]{HPR}\textup{)}; 
\item Split groups:
\subitem $G= \on{PGL}_2$, and any $\bar{\mu}$;
\subitem $G=\on{PGL}_n$, $n\geq 3$, and $\bar{\mu}=l\cdot \om_i$, $i \in \{1, n-1\}$ and $l\geq 2$;
\subitem $G=\on{PSp}_{2n}$, $n\geq 2$, and $\bar{\mu}$ quasi-minuscule;
\subitem $G=\on{SO}_{7}$, and $\bar{\mu}=\om_3$ \textup{(}not quasi-minuscule\textup{)};
\subitem $G=G_2$, and $\bar{\mu}$ quasi-minuscule;
\item Non-split groups:
\subitem $G= \on{PU}_3$, and any $\bar{\mu}$;
\subitem $G= \on{PU}_{2n+1}$, $n\geq 2$, and $\bar{\mu}$ quasi-minuscule;
\subitem $G=\on{PSO}_{2n+2}$, $n\geq 2$, and $\bar{\mu}$ quasi-minuscule; 
\subitem $G=\on{PU}_6$, and $\bar{\mu}=\om_3$ \textup{(}not quasi-minuscule\textup{)};
\subitem $G = \,^3D_{4,2}$, the `ramified triality', and $\bar{\mu}$ quasi-minuscule.  
\end{itemize}
\end{thm}

Note that $\on{PU}_4$ is isomorphic to the non-split $\on{PSO}_6$, and therefore the quasi-minuscule Schubert variety for $\on{PU}_4$ is rationally smooth as well. 


For the formulation of our next result, we introduce the following notion. The triple $(G,\bar{\mu},x)$ is called of {\em exotic smoothness} if $G\simeq \on{PU}_{2n+1}$ for some $n\geq 1$,  the element $\bar{\mu}\in X_*(T)_I^+\bslash \{0\}$ is quasi-minuscule, and $x$ corresponds up to $G(F)$-conjugation to an {\em almost  modular} lattice,  i.e., the lattice times a uniformizer is contained in  the dual of the lattice, with colength $1$. See \S\ref{Abs_Special_Sec} below for a more conceptual interpretation of the last condition in terms of the Bruhat-Tits building (in the terminology of $\S\ref{Abs_Special_Sec}$, the condition on $x$ above amounts to requiring that $x$ is special but not absolutely special).

The following result verifies a conjectural classification which Rapoport postulated in conversations with the second named author in 2010. 

\begin{thm}\label{sm_thm} The normalization $\tGr_{G,x}^{\leq \bar{\mu}}$ is smooth if and only if either $\bar{\mu}$ is minuscule or the triple $(G,\bar{\mu},x)$ is of exotic smoothness.
\end{thm}

We note that Schubert varieties are normal if $\on{char}(k)\nmid |\pi_1(G)|$ by \cite[Thm.~6.1]{PR08}, e.g., if the characteristic of $k$ is zero or sufficiently large.
However, there are non-normal Schubert varieties in general, e.g., the Schubert variety for $G=\PGL_2$ and quasi-minuscule $\bar\mu$ is non-normal if $\on{char}(k)=2$, cf.~\cite{HLR}.


If $G$ is non-split, the only pairs with minuscule coweights are $(\on{PU}_{2n},\omega_1)$ and $(\on{PSO}_{2n+2},\om_n)$, cf.~Remark \ref{connected_rmk}. These relate to the cases 1) and 3) of local models of exotic good reduction above. The remaining case 2) corresponds to the case of exotic smoothness. 

Our approach to the classification is as follows. We first classify all rationally smooth Schubert varieties, and for this the nature of $x$ is unimportant. We prove that ${\rm Gr}_{G, x}^{\leq \bar{\mu}}$ is rationally smooth if and only if the representation $V_{\bar{\mu}}$ of $(G^\vee)^I$ is weight-multiplicity-free (cf.\,Proposition \ref{ratl_sm_prop}).  For this, we use the ramified geometric Satake correspondence \cite{Zhu15, Ri16}.  Next, we use Howe's classification of all weight-multiplicity-free representations of simple simply connected groups (Theorem \ref{mult_thm}). Together with our list of all possibilities for the reductive groups $(G^\vee)^I$ for $G$ adjoint and absolutely simple (Lemma \ref{dual_list}), we are able to establish the list in Theorem \ref{ratl_sm_thm} of all such pairs $(G, \bar{\mu})$ such that ${\rm Gr}_{G, x}^{\leq \bar{\mu}}$ is rationally smooth, cf.~\S\ref{Proof_Ratl_Sm_Sec}.

Since the normalization of Schubert varieties is a finite, birational, universal homeomorphism by \cite[Prop.~3.1]{HRc}, the cohomological characterization  of rational smoothness (Proposition \ref{Hansen}) shows that the variety $\Gr_{G,x}^{\leq \bar{\mu}}$ is rationally smooth if and only if its normalization ${\tGr}_{G,x}^{\leq \bar{\mu}}$ is rationally smooth. 
In particular, we obtain the same list of rationally smooth {\em normalized Schubert varieties}.
The remaining work is to determine which ${\tGr}_{G,x}^{\leq \bar{\mu}}$ on this list are smooth. The proof is given in \S\ref{Proof_Sm_Sec} below. For split groups, we only rely on the case of $\PGL_2$, the Levi Lemma and the quasi-minuscule cases of \cite{MOV05}, and hence we do not rely on computer aided calculations. For the non-split case, we rely on a few calculations for classical groups from \cite{P00, Arz09, PR09, HPR}. 
Here we use that $\tGr_{G,x}^{\leq \bar{\mu}}$ is isomorphic to a Schubert variety for a suitable central extension $\tilde G\to G$ in order to apply these results.

The most difficult case in our proof is the quasi-minuscule Schubert variety for the {\em ramified triality}, i.e., the non-split form of ${\rm Spin}_8$. This case is studied in $\S\ref{Triality_Sec}$, and it is also used by \cite[Thm 1.2]{HPR} to rule out the possibility of additional cases of exotic good reduction. The ramified triality plays a special role, in that it is not amenable to the methods in \cite{HPR}.

Let us remark that the results in the split and ${\rm char}(k)=0$ context \cite{EM99, MOV05} are stronger: the smooth locus of $\Gr_{G,x}^{\leq\bar{\mu}}$ is exactly the open stratum $\Gr_{G,x}^{\bar{\mu}}$. Due to the phenomenon of exotic smoothness, this fails in the non-split case. In \S\ref{Conjecture_Sec}, we formulate a conjecture which describes the precise conditions on $x$ needed to ensure that this description of the smooth locus holds.

In light of Theorem \ref{sm_thm}, in order to give a classification of smooth Schubert varieties, it suffices to understand which $\Gr^{\leq \bar{\mu}}_{G,x}$ are normal.  We plan to address this question in \cite{HLR}.

\medskip
\noindent\textbf{Acknowledgements.} It is a pleasure to thank Michael Rapoport for his steady encouragement and interest in our work, and for all he has taught us over the years. We also thank him for his detailed comments on the article. We thank Johannes Ansch\"utz, Xuhua He, George Pappas, and Brian Smithling for interesting discussions around the subject of this article, Jeff Adams for his help with LiE, and Mark Reeder for pointing us to the reference \cite{BZ90}. Finally, we express our gratitude to David Hansen for his very helpful suggestions and remarks. {

\section{Rational smoothness of Schubert varieties}\label{Ratl_Smooth_Sec}
Let $k$ be an algebraically closed field, and let $F=k\rpot{t}$ denote the Laurent series field. Let $G$ be a connected reductive group over $F$ which splits over a tamely ramified Galois extension $F'/F$. Denote $I=\Gal(F'/F)$. Let $x\in \scrB(G,F)$ be a special vertex in the Bruhat-Tits building, and denote by $\Gr_{G,x}:=LG/L^+\calG_x$ the twisted affine Grassmannian in the sense of \cite{PR08}. Let $S\subset G$ be a maximal $F$-split torus such that $x$ belongs to the apartment $\scrA(G,S,F)$, cf.~\cite[Thm 7.4.18 (i)]{BT72}. The centralizer $T=Z_G(S)$ is a maximal torus defined over $F$ (because by Steinberg's theorem $G$ is quasi-split). Let $B\subset G$ be a Borel subgroup containing $T$, and defined over $F$.

We equip the coinvariants $X_*(T)_I$ with the dominance order $\leq$ with respect to the \'echelonnage root system $\breve{\Sig}$, cf.~\cite{Hai18}. We denote by $X_*(T)_I^+\subset X_*(T)_I$ the submonoid of dominant elements. 
(One can show that $X_*(T) \to X_*(T)_I$ induces a {\em surjective} map of monoids $X_*(T)^+ \to X_*(T)_I^+$.) 
For each $\bar{\mu}\in X_*(T)_I^+$, we have the Schubert variety $\Gr_{G,x}^{\leq\bar{\mu}}\subset \Gr_{G,x}$, and the open orbit embedding $j_{\bar{\mu}}\co \Gr_{G,x}^{\bar{\mu}}\hookrightarrow \Gr_{G,x}^{\leq\bar{\mu}}$. For $\bar{\mu}, \bar{\la}\in X_*(T)_I^+$, we have $\Gr_{G,x}^{\bar{\la}}\subset \Gr_{G,x}^{\leq\bar{\mu}}$ if and only if $\bar{\la}\leq \bar{\mu}$ in the dominance order, cf.~\cite[Cor 1.8, Prop 2.8]{Ri13}. 

Since the $I$-action preserves a pinning of $G^\vee$, the group $(G^\vee)^I$ is a possibly disconnected reductive $\algQl$-group, cf.~\cite[Prop 4.1(a)]{Hai15}. There is a unique (up to isomorphism) irreducible representation $V_{\bar{\mu}}$ of $(G^\vee)^I$ with highest $(T^\vee)^I$-weight $\bar{\mu}$, cf.\,\cite[Lem.\,4.10]{Zhu15}, \cite[Cor.\,A.9]{Ri16}, \cite[$\S5.2$]{Hai18}. For each $\bar{\la}\leq \bar{\mu}$, we denote by $d_{\bar{\mu}}(\bar{\la})$ the dimension of the $\bar{\la}$-weight space $V_{\bar{\mu}}(\bar{\la})$.

Fix a prime number $\ell$ coprime to $\on{char}(k)$. Denote $d_{\bar{\mu}}=\dim(\Gr_{G,x}^{\leq \bar{\mu}})$. The intersection complex $\IC_{\bar{\mu}} =j_{\bar{\mu},!*}\algQl[d_{\bar{\mu}}]$ corresponds under the ramified geometric Satake isomorphism \cite{Zhu15, Ri16} to the irreducible $\algQl$-representation $V_{\bar{\mu}}$ of $(G^\vee)^I$. 

Recall that an irreducible variety $Y$ of dimension $d$ over $k$ is called {\em $\ell$-rationally smooth} if for every point $y\in Y(k)$ with closed immersion $i_y\co \Spec(k) \hookrightarrow Y$, there is an isomorphism
$$
i^!_y \, \algQl \cong \mathbb \algQl[-2d]
$$
in the derived category $D^b_c(\{y\}, \algQl)$. This notion coincides with the one given in \cite[Def.\,A.1]{KL79} since ${\mathbb H}^n_{y}(Y, \algQl) := {\mathbb H}^n(Y, i_{y,*}i^!_y \algQl) = H^n(i^!_y \, \algQl)$. Further, we say $y \in Y(k)$ is an $\ell$-rationally smooth point of $Y$ if $y$ is contained in an $\ell$-rationally smooth Zariski open subset of $Y$. Therefore by definition the $\ell$-rationally smooth locus is open in $Y$. It is clear that every smooth variety is $\ell$-rationally smooth. But there exist many non-smooth, but $\ell$-rationally smooth varieties.

We use the following characterization of $\ell$-rational smoothness which was explained to us by David Hansen.

Let $p: Y \rightarrow {\rm Spec}(k)$ be the structure morphism, and consider the Verdier dualizing complex $\omega_Y := p^!\algQl$. Denote by $\mathbb D_Y(\mathcal F) = R{\mathcal Hom}_{D^b_c(Y)}(\mathcal F, \omega_Y)$, where $\mathcal F$ belongs to $D^b_c(Y)$, the derived category of bounded constructible $\algQl$-complexes on $Y$. It follows that $\omega_Y = \mathbb D_Y(\bar{\bbQ}_{\ell})$.  

\begin{prop} \textup{(Hansen)} \label{Hansen}
The following statements are equivalent:
\begin{enumerate}
\item[i)] $Y$ is $\ell$-rationally smooth;
\item[ii)] $\omega_Y \simeq \bar{\bbQ}_{\ell}[2d]$;
\item[iii)] ${\rm IC}_Y \simeq \bar{\bbQ}_{\ell}[d]$. 
\end{enumerate}
\end{prop}

\begin{proof}
The implications iii) $\Rightarrow $ ii) $\Rightarrow$ i) are straightforward. We abbreviate by writing $A := \bar{\bbQ}_{\ell}$. For ${\rm i)} \Rightarrow  {\rm ii)}$, using $\mathbb D_Y(A) = \omega_Y$ we note that for any closed point $y$, the stalk $i_y^!A$ is the dual of $i_y^*\omega_Y$. Thus by i), the complex $\omega_Y$ is concentrated in degree $-2d$ and $(H^{-2d}(\omega_Y))_y \simeq A$. This forces  any choice of non-zero map $A[2d] \to \omega_Y$ to be an isomorphism. Note that such non-zero maps exist because ${\rm Hom}_{D^b_c(Y)}(A[2d], \omega_Y) = {\mathbb H}^{-2d}(Y,\omega_Y)$ is dual to ${\bbH}^{2d}_c(Y, A) \simeq A$.

For ${\rm ii)} \Rightarrow {\rm iii)}$, we can choose maps $A[d] \to \IC_Y \to \omega_Y[-d]$ which are isomorphisms on a dense open subset (choose any non-zero map $A[d] \to \IC_Y$ using ${\rm Hom}_{D^b_c(Y)} (A[d], \IC_Y) = \bbH^{-d}(Y, \IC_Y)$ is dual to $\bbH^d_c(Y, \IC_Y) \simeq A$\footnote{For any non-empty open subset $U\subset Y$ with closed complement $Z \subset Y$, the natural map $\bbH^d_c(U,\IC_Y|_U)\to \bbH^d_c(Y,\IC_Y)$ is an isomorphism. This follows from $\IC_Y|_Z \in \,^pD^{\leq -1}(Z)$ and the estimate of middle-perverse cohomological amplitude $p_! :\, \leq d-1$ for $p: Z \to \Spec(k)$ (see \cite[4.2.4]{BBD82}).
We apply this to any non-empty open subset $U\subset Y$ with $\IC_Y|_U=A[d]$.}, and choose $\IC_Y \to \omega_Y[-d]$ by taking the Verdier dual of the first map).  Now ii) guarantees that $A[d] \cong \omega_Y[-d]$ is a perverse sheaf, and the aforementioned maps split $A[d]$ off as a direct summand of $\IC_Y$.  Since $\IC_Y$ is a simple perverse sheaf, this implies iii).
\end{proof}

The following proposition is proved using the ramified geometric Satake correspondence \cite{Zhu15, Ri16} as well as elaborations on it such as \cite[Thm 5.1]{Zhu15}. 

Recall that the ramified geometric Satake equivalence provides an equivalence of Tannakian categories
\[
\on{Perv}_{L^+\calG}(\Gr_{G,x})\;\simeq\; \Rep_{\algQl}\big((G^\vee)^I\big),
\]
under which the intersection complex $\IC_{\bar\mu}$ corresponds to the representation $V_{\bar\mu}$. 
The left hand side is the category of $L^+\calG$-equivariant perverse sheaves on $\Gr_{G,x}$ equipped with the tensor structure given by the convolution product and the fibre functor given by global cohomology.
The $\ell$-rational smoothness of $\Gr_{G,x}^{\leq \bar{\mu}}$ is related to the structure of the $(G^\vee)^I$-representation $V_{\bar\mu}$ as follows.

\begin{prop}\label{ratl_sm_prop} The following are equivalent:\smallskip\\
i\textup{)} The Schubert variety $\Gr_{G,x}^{\leq \bar{\mu}}$ is $\ell$-rationally smooth. \smallskip\\
ii\textup{)} The intersection complex $\IC_{\bar{\mu}}$ is isomorphic to the constant sheaf $\algQl[d_{\bar{\mu}}]$. \smallskip\\
iii\textup{)} One has $d_{\bar{\mu}}(\bar{\la})=1$ for all $\bar{\la}\in X_*(T)_I^+, \bar{\la}\leq \bar{\mu}$.\smallskip\\
In particular, the rational smoothness of $\Gr_{G,x}^{\leq \bar{\mu}}$ is independent of the choice of $\ell$ and of the choice of special vertex $x$. Thus, we replace ``$\ell$-rationally smooth'' from now on by ``rationally smooth'' for Schubert varieties.
\end{prop}
\begin{proof} 
Denote by $\calG=\calG_x$ the special parahoric group scheme. Write $S_{\bar{\mu}}:={\rm Gr}^{\leq \bar{\mu}}_{G,x}$, and $C_{\bar{\la}}:=\Gr_{G,x}^{\bar{\la}}$ for any $\bar{\la}\in X_*(T)_I^+, \bar{\la}\leq \bar{\mu}$. Let $i_{\bar{\la}}\co \{x_{\bar{\la}}\} \hookrightarrow S_{\bar{\mu}}$ be the closed immersion of the base point $x_{\bar{\la}}\in C_{\bar{\la}}(k)$ corresponding to $\bar{\la}$.\smallskip\\
i)$\Rightarrow$ ii): This is a special case of Proposition \ref{Hansen}.\\  
\iffalse We can also argue without that proposition, as follows.} Assume that $S_{\bar{\mu}}$ is rationally smooth. Since $k$ is algebraically closed, the points $S_{\bar{\mu}}(k)\subset S_{\bar{\mu}}$ are everywhere dense. It follows from the definition of the selfdual perverse $t$-structure on $D_c^b(S_{\bar{\mu}},\algQl)$ (cf.~\cite[\S III.1]{KW01}) that the constant sheaf $\algQl[d_{\bar{\mu}}]$ supported on $S_{\bar{\mu}}$ is perverse (because the dual of $\algQl[d_{\bar{\mu}}]$ is concentrated in degree $d_{\bar{\mu}}$). As it is $L^+\calG$-equivariant as well, it is an object in the category of $L^+\calG$-equivariant perverse sheaves on ${\rm Gr}_{G,x}$. By \cite[Lem 1.1]{Zhu15}, \cite[Thm 4.2 (iii)]{Ri16}, this category is semi-simple, and hence $\IC_{\bar{\mu}}$ is a direct summand of $\algQl[d_{\bar{\mu}}]$. Comparing the cohomology stalks implies $\IC_{\bar{\mu}}=\algQl[d_{\bar{\mu}}]$, and proves ii). Namely, write $\algQl[d_{\bar{\mu}}] = \IC_{\bar{\mu}} \oplus \mathcal E$, where $\mathcal E$ is another $L^+\calG$-equivariant perverse sheaf on $\Gr_{G,x}$. Restriction to $C_{\bar{\mu}}$ shows that $\mathcal E$ is supported on a union of sets $C_{\bar{\lambda}}$ for $\lambda \in X_*(T)_I^+$ and $\bar{\lambda} < \bar{\mu}$, hence being a sum of sheaves $\IC_{\bar{\lambda}}$ with $\bar{\lambda} < \bar{\mu}$, it lives in cohomological degrees greater than $-d_{\bar{\mu}}$. But also $H^i\mathcal E = 0$ for all $i > -d_{\bar{\mu}}$, hence $\mathcal E = 0$. \fi
ii)$\Rightarrow$iii): Assume that $\IC_{\bar{\mu}}=\algQl[d_{\bar{\mu}}]$. By setting $q=1$ in \cite[Thm 5.1]{Zhu15}, we see that the dimension of the total cohomology of $i^*_{{\bar{\lambda}}}\IC_{\bar{\mu}}=\algQl[d_{\bar{\mu}}]$ is equal to $d_{\bar{\mu}}(\bar{\la})$ for all $\bar{\la}\in X_*(T)_I^+, \bar{\la}\leq \bar{\mu}$. This is equal to $1$, which proves iii).\smallskip\\
iii)$\Rightarrow$i): Assume that $d_{\bar{\mu}}(\bar{\lambda}) = 1$ for all $\bar{\lambda} \in X_*(T)_I^+$ with $\bar{\lambda} \leq \bar{\mu}$. By \cite[Thm.\,5.1, Prop.\,5.4]{Zhu15}, we have
\begin{equation} \label{Zhu_5.1}
i_{\bar{\la}}^*~{\rm IC}_{\bar{\mu}} = \algQl[d_{\bar{\mu}}] \hspace{.5in} i_{\bar{\la}}^! ~{\rm IC}_{\bar{\mu}} = \algQl[-d_{\bar{\mu}}].
\end{equation}
Indeed, \cite[Thm,\,5.1]{Zhu15} implies that $\dim(H^*(i_{\bar{\lambda}}^*~{\rm IC}_{\bar{\mu}}))=1$, and then in conjunction with \cite[Prop.\,5.4]{Zhu15} we see that $\dim(H^{-d_{\bar{\mu}}}(i_{\bar{\lambda}}^*~{\rm IC}_{\bar{\mu}}))=1$, which yields the first formula. The second formula follows from the first by applying Verdier duality.

Since any point in $S_{\bar{\mu}}(k)$ lies in the $L^+\calG$-orbit of some base point $x_{\bar{\lambda}} \in C_{\bar{\lambda}}(k)$, this implies that $\IC_{\bar{\mu}}=\calF[d_{\bar{\mu}}]$ where $\calF:=H^{-d_{\bar{\mu}}}(\IC_{\bar{\mu}})$ is an $L^+\calG$-equivariant constructible $\algQl$-sheaf. Here we are using the principle that if a complex $K \in D^b_c(S_{\bar{\mu}},\algQl)$ is cohomologically supported in degree $n \in \mathbb Z$, then $K = H^{n}K[-n]$ in $D^b_c(S_{\bar{\mu}}, \algQl)$. 

Hence, to prove that $S_{\bar{\mu}}$ is rationally smooth it is by \eqref{Zhu_5.1} enough to prove that
\begin{equation}\label{constant}
\IC_{\bar{\mu}} \,=\, \algQl[d_{\bar{\mu}}],
\end{equation}
or equivalently $\calF= \algQl$. Let $D$ denote the derived category $D_c^b(S_{\bar{\mu}}, \algQl)$ and write ${\bbH}(K)$ for the global cohomology of an object $K \in D$. We have
\begin{equation}\label{Derived_Hom}
\Hom_D(\bar{\bbQ}_\ell[d_{\bar{\mu}}], \IC_{\bar{\mu}})\,=\,\bbH^{-d_{\bar{\mu}}}(\IC_{\bar{\mu}})\,=\,\bbH^0(\calF),
\end{equation}
and this vector space corresponds under geometric Satake to the $1$-dimensional lowest weight space of $V_{\bar{\mu}}$. We claim that any vector $v\in \bbH^{-d_{\bar{\mu}}}(\IC_{\bar{\mu}})\bslash \{0\}$ induces an isomorphism 
\begin{equation}\label{IC_Iso}
\iota_v\co \bar{\bbQ}_\ell[d_{\bar{\mu}}]\, \overset{\sim}{\longrightarrow} \,\IC_{\bar{\mu}}=\calF[d_{\bar{\mu}}].
\end{equation}
 The map $\iota_v$ is necessarily an isomorphism on a dense open subset by (\ref{Zhu_5.1}). Hence the kernel of $\iota_v[-d_{\bar{\mu}}]$ is a subsheaf of $\bar{\bbQ}_\ell$ which is supported on a nowhere dense closed subset; any such subsheaf of $\bar{\bbQ}_\ell$ is zero. Therefore $\iota_v$ is an injective morphism of constructible abelian sheaves, which is an isomorphism on the stalks at all closed points, since by (\ref{Zhu_5.1}) $\mathcal F$ has 1-dimensional stalks everywhere. This implies $\iota_v$ is an isomorphism and completes the proof.

\end{proof}

\begin{rmk} \label{ratl_sm_split_rmk}
If $G$ is split, there is a sharper stratum-by-stratum version of Proposition \ref{ratl_sm_prop}.  Suppose $\lambda, \mu \in X_*(T)^+$ satisfy $\lambda \leq \mu$. Then $x_{\lambda}$ belongs to the rationally smooth locus of $S_{\mu}$ if and only if $d_{\mu}(\lambda)=1$. This is well-known, but for completeness we explain the proof here. Let $W$ (resp. $W_x$, resp.\,$\Fl_G$) denote the Iwahori-Weyl group (resp.\,finite Weyl group, resp.\,affine flag variety) for $G$ relative to the special vertex $x$ and an alcove ${\bf a}$ containing $x$ in its closure. Let $w_\lambda \in W$ denote the unique longest element in $W_x \lambda(t) W_x$.  Let $\tilde{x}_\lambda \in \Fl_G$ denote the base point in the Iwahori orbit corresponding to $w_\lambda$. Then as $\pi: \Fl_G \to \Gr_{G,x}$ is represented by a smooth surjective $L^+\calG$-equivariant morphism, $x_{\lambda}$ belongs to the rationally smooth locus of $S_{\mu}$ if and only if $\tilde{x}_\lambda$ belongs to the rationally smooth locus of the Schubert variety $S_{w_\mu} := \pi^{-1}(S_\mu)$ in $\Fl_G$. By \cite[Thm.\,A.2]{KL79}, this is equivalent to the triviality of certain  Kazhdan-Lusztig polynomials, namely $P_{w', w_\mu}(q) = 1$ for all $w' \in W$ with $w_\lambda \leq w' \leq w_\mu$ (in loc.\,cit.\,this is proved for Schubert varieties in the classical flag variety for a split group, but the proof carries over to the affine flag varieties).  This is equivalent to the single equality $P_{w_\lambda, w_\mu}(q)=1$ (e.g.\,\cite[Thm.\,6.2.10]{BL00}, using that all $P_{u,v}(q)$ have non-negative coefficients for $u,v \in W$ by \cite{KL80}).  Since $P_{u,v}(0)=1$ (e.g,\,\cite[Lem.\,6.1.9]{BL00}), the equality is equivalent to $P_{w_\lambda, w_\mu}(1)=1$. Finally, this is equivalent to $d_{\mu}(\lambda) = 1$ by Lusztig's multiplicity formula $P_{w_\lambda, w_\mu}(1) = d_{\mu}(\lambda)$ (\cite[Thm.\,6.1]{Lu83}). (Because Kazhdan-Lusztig polynomials $P_{x,y}(q^{1/2}) \in \mathbb Z[q^{1/2}]$ attached to Hecke algebras with unequal parameters are not known to belong to $\mathbb Z_{\geq 0}[q]$, cf.\,\cite{Lus03}, it is not clear that the same argument can be used to handle quasi-split but non-split groups.) We remark that Berenstein-Zelevinsky \cite{BZ90} have classified, for any connected reductive complex group, all pairs of weights $(\mu, \lambda)$ satisfying $d_{\mu}(\lambda) = 1$. 
\end{rmk}

\section{The classification for reductive groups: passing to adjoint groups}
We proceed with the notation of \S\ref{Ratl_Smooth_Sec}. Let $G\to G_\ad$ be the canonical map to the adjoint group, and denote by $T_\ad\subset B_\ad$ the image of $T\subset B$. The image of the special vertex $x$ under $\scrB(G,F)\to \scrB(G_\ad,F)$ defines a special vertex $x_\ad$. The map $G\to G_\ad$ extends to a map of parahoric $\calO_F$-groups $\calG_x\to \calG_{x_\ad}$. By functoriality of the loop group construction, we obtain a map $LG\to LG_\ad$ (resp.~$L^+\calG_x\to L^+\calG_{x_\ad}$), and hence a map on twisted affine Grassmannians $\Gr_{G,x}\to \Gr_{G_\ad,x_\ad}$. 

Further, $T\to T_\ad$ defines a map $X_*(T)_I\to X_*(T_\ad)_I$ which sends $X_*(T)_I^+$ to $X_*(T_\ad)_I^+$. For $\bar{\mu}\in X_*(T)_I^+$, we denote by $\bar{\mu}_\ad\in X_*(T_\ad)_I^+$ its image. Since the Schubert varieties are defined as the scheme theoretic image of the orbit map, we get a natural morphism of $k$-schemes
\begin{equation}\label{Schubert_ad}
\Gr_{G,x}^{\leq \bar{\mu}}\,\longto\, \Gr_{G_\ad,x_\ad}^{\leq \bar{\mu}_\ad}.
\end{equation}

\begin{prop} \label{Ad_Iso_prop}
The map \eqref{Schubert_ad} is a finite, birational, universal homeomorphism. In particular, it induces an equivalence of \'etale sites and an isomorphism on normalizations.
\end{prop}
\begin{proof} 
This is a special case of \cite[Prop.~3.5]{HRc}. The equivalence on \'etale sites follows from their topological invariance \cite[04DY]{StaPro}.
\end{proof}

Now assume that $G=G_\ad$ is adjoint. Then there is a finite index set $J$, and an isomorphism of $F$-groups 
\begin{equation}\label{Ad_Iso}
G=\prod_{j\in J}\Res_{F_j/F}(G_j), 
\end{equation}
where each $F_j/F$ is a finite separable field extension, and $G_j$ is an absolutely simple adjoint $F_j$-group. The condition on $G$ of being tamely ramified implies that each $G_j$ is tamely ramified (and likewise for $F_j/F$, but this is not important as we will see). This induces an identification of buildings $\scrB(G,F)=\prod_{j\in J}\scrB(G_j,F_j)$ compatible with the simplicial structure, cf.~\cite[Prop 4.6]{HRb}. Under this identification we get $x=(x_j)_{j\in J}$ where each vertex $x_j\in \scrB(G_j,F_j)$ is special. 

Further, we can write $T=\prod_{j\in J}\Res_{F_j/F}(T_j)$, and likewise for $B$, cf.~\cite[Lem 4.2]{HRb}. Note that the splitting field $F'$ of $G$ contains each $F_j$, and we define $I_j:=\Gal(F/F_j)$. By Shapiro's Lemma (cf.~\cite[Lem 4.1]{HRb}), we get $X_*(T)_I=\prod_{j\in J}X_*(T_j)_{I_j}$ compatible with $X_*(T)_I^+=\prod_{j\in J}X_*(T_j)_{I_j}^+$. For each $\bar{\mu}\in X_*(T)_I^+$, we denote $\bar{\mu}=(\bar{\mu}_j)_{j\in J}$ with $\bar{\mu}_j\in X_*(T_j)_{I_j}^+$.

\begin{lem}\label{product}
Under \eqref{Ad_Iso} there is an identification of affine Grassmannians
\[
\Gr_{G,x}\,=\,\prod_{j\in J} \Gr_{G_j,x_j},
\]
under which the Schubert varieties \textup{(}resp.~their normalizations\textup{)} for each $\bar{\mu}=(\bar{\mu}_j)_{j\in J}$ correspond to each other.
\end{lem}
\begin{proof}
It is enough to treat the following two cases separately.\smallskip\\
\textit{Products:} If $G=G_1\times G_2$ is a direct product of two $F$-groups, then we have $\Gr_{G,x}=\Gr_{G_1,x_1}\times \Gr_{G_2,x_2}$ which is obvious. Also the equality $\Gr_{G,x}^{\leq \bar{\mu}}=\Gr_{G_1,x_1}^{\leq \bar{\mu}_1}\times \Gr_{G_2,x_2}^{\leq\bar{\mu}_2}$ is easy to prove using that the product of (geometrically) reduced $k$-schemes is reduced, cf.~\cite[035Z (2)]{StaPro}.
Likewise, the equality holds on normalizations using that the product of (geometrically) normal $k$-schemes is normal, cf.~\cite[06DG]{StaPro}.\smallskip\\
\textit{Restriction of scalars:} Let $G=\Res_{F'/F}(G')$ where $F'/F$ is a finite separable extension, and $G'$ is an $F'$-group. By \cite[Prop 4.7]{HRb}, we have $\calG_x=\Res_{\calO_{F'}/\calO_F}(\calG'_{x'})$ where we use the identification $\scrB(G,F)=\scrB(G',F')$. Now choose\footnote{The identification of twisted affine Grassmannians is independent of this choice as all loop groups can be defined without choosing uniformizers, cf.~\cite[\S 2]{Ri13}.} a uniformizer $u\in \calO_{F'}$. Since $k$ is algebraically closed, we have $\calO_{F'}=k\pot{u}$ (resp.~$F'=k\rpot{u}$). For any $k$-algebra $R$, we have $R\pot{t}\otimes_{\calO_{F}}\calO_{F'}=R\pot{u}$ (resp.~$R\rpot{t}\otimes_{{F}} {F'}=R\rpot{u}$). This gives an equality on loop groups $L^+\calG_x=L^+\calG'_{x'}$ (resp.~$LG=LG'$). Hence, there is an equality on twisted affine Grassmannians $\Gr_{G,x}=\Gr_{G',x'}$, and it is obvious that the Schubert varieties (resp.~their normalizations) correspond to each other.
\end{proof}

Combining Proposition \ref{Ad_Iso_prop} with Lemma \ref{product}, it is obvious how to extend our classification from the absolutely simple adjoint case to the case of general tamely ramified connected reductive groups. From the discussion, we also see that we can relax the condition on $G$ of being tamely ramified to the condition that each absolutely simple adjoint factor $G_i$ is tamely ramified. In particular, our classification includes all cases where $\on{char}(k) \geq 5$, cf.~the beginning of the next section.

\section{Weight-multiplicity-free representations}\label{weight_free_sec}
We proceed with the notation of \S\ref{Ratl_Smooth_Sec}, and assume further that $G$ is adjoint and absolutely simple. Then the splitting field $F'/F$ is of degree $[F':F]=1$, $2$ or $3$, cf.~\cite[\S 4]{Ti77}. So if $G$ is non-split (i.e., $[F':F]=2$ or $3$), then the assumption of being tame excludes only $2$ or $3$ from being the residue characteristic $[F':F]$.

We are interested in classifying all irreducible representations $V_{\bar{\mu}}$ of $(G^\vee)^I$ such that $d_{\bar{\mu}}(\bar{\la})=1$ for all $\bar{\la}\in X_*(T)_I^+, \bar{\la}\leq \bar{\mu}$. These representations are called {\em weight-multiplicity-free}.

\begin{lem} \label{dual_grp}
The group $(G^\vee)^I$ is a connected reductive $\algQl$-group which is simple and semisimple. Furthermore, it is simply connected except in the case $G^\vee= \on{SL}_{2n+1}$, $n\geq 1$ with a non-trivial $I$-action, in which case $(G^\vee)^I \cong \on{SO}_{2n+1}$.
\end{lem}
\begin{proof} 
Since the $I$-action preserves a pinning of $G^\vee$, this follows from Proposition \ref{fix_pt_prop}, taking $\kappa = \algQl$.
\end{proof}

\begin{lem}\label{dual_list}
The following list gives all possibilities for $(G^\vee)^I$:
\begin{itemize}
\item[i)] $[F':F]=1$: $G$ split; $(G^\vee)^I=G^\vee$.
\item[ii)] $[F':F]=2$: 
\subitem a\textup{)} $G =\on{PU}_{2n}$, $n\geq 3$, name $B\str C_n$; $(G^\vee)^I=\on{Sp}_{2n}$, type $C_n$.
\subitem b\textup{)} $G =\on{PU}_{2n+1}$, $n\geq 1$, name $C\str BC_n$; $(G^\vee)^I=\on{SO}_{2n+1}$, type $B_n$. 
\subitem \textup{c)} $G =\on{PSO}_{2n+2}$, $n\geq 2$, name $C\str B_n$; $(G^\vee)^I = {\rm Spin}_{2n+1}$, type $B_n$. 
\subitem \textup{d)} $G  = \,^2E^2_{6,4}$,  `ramified $E_6$', name $F_4^I$; $(G^\vee)^I=F_4$.
\item[iii)] $[F':F]=3$: $G = \,^3D_{4,2}$, `ramified triality', name $G_2^I$; $(G^\vee)^I=G_2$.
\end{itemize}
Hence, $(G^\vee)^I$ is simply connected except in case ii.b\textup{)} where the connection index is $2$. In this case, $(G^\vee)^I=\on{SO}_{2n+1}$ is adjoint.
\end{lem}

Here, the `name' refers to the name given by Tits in \cite[Table 4.2]{Ti77}, and the `type' refers to the type of the connected reductive group $(G^\vee)^I$. Tables containing essentially this content are contained in \cite[$\S5.2$]{HPR}, but here we describe the groups in classical terms, including the isogeny type.

\begin{rmk}\label{connected_rmk}
Case ii.b) shows that there is no nonzero $\breve{\Sigma}$-minuscule coweight for the non-split group $\on{PU}_{2n+1}$.  The fact that $\on{SO}_{2n+1}$ is adjoint means that every weight is in the root lattice.  This translates to  $X_*(T)_I = \bbZ[\breve{\Sigma}^\vee]$, which  in turn implies that the affine Grassmannian for $\on{PU}_{2n+1}$ is connected. Similarly, one proves that the affine Grassmannian for a non-split absolutely simple adjoint group $G$ is always connected, except in cases ii.a) and ii.c) where it has two connected components. This also shows that only these cases admit minuscule elements: checking the tables in \cite[Planche II \& III]{Bou} identifies $\om_1$ in ii.a), and $\om_n$ in ii.c) as the minuscule elements.
\end{rmk}

\begin{proof}
Checking the tables in \cite[\S 4]{Ti77} for residually split groups gives the above list. We make the following remarks.  Tits' tables list the \'{e}chelonnage root system attached to $G/F$. For example the group named $B \str C_n$ is a ramified unitary group ${\rm PU}_{2n}$ and has \'{e}chelonnage root system $\breve{\Sigma}$ of type $B_n$.  The group $(G^\vee)^I$ has type dual to $\breve{\Sigma}$ (cf.\,\cite[$\S5.1$]{Hai18}), and thus has type $C_n$ and being simply connected (Lemma \ref{dual_grp}), we see that $(G^\vee)^I = {\rm Sp}_{2n}$.  The other cases are handled similarly.
\end{proof}

The following theorem is proven in R.\,Howe's article \cite[Thm 4.6.3]{Ho95}, and we refer the reader to its introduction for further references on the subject. A classification of multiplicity one {\em primitive} pairs $\bar{\la}<\bar{\mu}$ is also given in \cite{BZ90}: these are the pairs such that $d_{\bar{\mu}}(\bar{\lambda}) = 1$ and every simple root for $(G^\vee)^I$ appears at least once in the difference $\bar{\mu} - \bar{\lambda}$; from this one may classify all pairs such that $d_{\bar{\mu}}(\bar{\lambda}) = 1$.

\begin{thm} \label{mult_thm}
Let $\bar{\mu}\in X_*(T)_I^+$, and denote by $X_n$ the type of $(G^\vee)^I$ where $n\geq 1$ is the rank of $(G^\vee)^I$. Then the $(G^\vee)^I$-representation $V_{\bar{\mu}}$ is weight-multiplicity-free if and only if the pair $(X_n,\bar{\mu})$ appears in the following list:
\begin{itemize}
\item any type $X_n$, and $\bar{\mu}$ minuscule;
\item type $A_1$, and $\bar{\mu}$ arbitrary;
\item type $A_n$, $n\geq 2$, and $\bar{\mu}=l\cdot \om_i$ for $i \in \{1, n\}$ and $l\geq 2$;
\item type $B_n$, $n\geq 2$, and $\bar{\mu}$ quasi-minuscule;
\item type $C_3$, and $\bar{\mu}=\om_3$ \textup{(}not quasi-minuscule\textup{)}; 
\item type $G_2$, and $\bar{\mu}$ quasi-minuscule.
\end{itemize}
\qed
\end{thm}

\begin{rmk}
It is interesting to observe that only in type $A_n$ are there infinitely many weight-multiplicity-free representations. Also, outside of type $A$ cases and the single type $C_3$ case, the following implication holds: ``if $V_{\bar{\mu}}$ is weight-multiplicity-free, then $\bar{\mu}$ is (quasi-)minuscule.''
\end{rmk}

\subsection{Proof of Theorem \ref{ratl_sm_thm}}\label{Proof_Ratl_Sm_Sec}
This is a combination of Proposition \ref{ratl_sm_prop} and Theorem \ref{mult_thm} with the list in Lemma \ref{dual_list}. Indeed, these results make no reference to the choice of special vertex $x\in \scrB(G,F)$ which we therefore do not specify. Drop it from the notation for the rest of the proof. By Proposition \ref{ratl_sm_prop}, the Schubert variety $\Gr_{G}^{\leq \bar{\mu}}$ is rationally smooth if and only if the $(G^\vee)^I$-representation $V_{\bar{\mu}}$ is weight-multiplicity-free. Theorem \ref{mult_thm} gives a complete list of all pairs $((G^\vee)^I,\bar{\mu})$ such that $V_{\bar{\mu}}$ is weight-multiplicity-free. Clearly, if $\bar{\mu}$ is minuscule, there are no restrictions on the group. Assume now that $\bar{\mu}$ is not minuscule. If $G$ is split, then Theorem \ref{mult_thm} directly applies to give the rationally smooth cases listed in our theorem. If $G$ is not split, we use Lemma \ref{dual_list} to translate Theorem \ref{mult_thm} back in terms of the group $G$. Note that the group $\on{PU}_3$ appears from Lemma \ref{dual_list}, ii.b) using the exceptional isomorphism of Lie types $B_1=A_1$. This proves the theorem. 

\begin{rmk} \label{zero_wt_rmk}
When $\bar{\mu}$ is quasi-minuscule, it is known that the dimension of the zero-weight space $V_{\bar{\mu}}(0)$ is the number of short nodes in the Dynkin diagram for the group $(G^\vee)^I$. From this, Lemma \ref{dual_list}, and Proposition \ref{ratl_sm_prop}, one can easily determine the groups $G$ whose quasi-minuscule Schubert variety is rationally smooth, without invoking Theorem \ref{mult_thm}.
\end{rmk}

\section{Absolutely special vertices}\label{Abs_Special_Sec}

Temporarily we assume $G$ is any connected reductive group over an arbitrary field $F$ endowed with a non-trivial discrete valuation, and $F'/F$ is a finite separable extension splitting $G$. Following \cite{Ti77} we assume $F$ is complete and its residue field is perfect.

\begin{dfn}
A vertex $x\in \scrB(G,F)$ is called {\em absolutely special} if its image under the simplicial embedding $\scrB(G,F)\hookto \scrB(G,F')$ is a special vertex.
\end{dfn}
Note that this notion is independent of the choice of the splitting field $F'/F$.

\begin{lem} \label{abs_spec_lem}
Absolutely special vertices exist in every quasi-split group $G$ and are special. 
\end{lem}
\begin{proof} 
This is modeled on Tits' proof of the existence of hyperspecial points for unramified groups, cf.\,\cite[p.\,36]{Ti77}. Since $G$ is quasi-split, there exist $S\subset T\subset B$ defined over $F$ as above. We may assume $F'/F$ is Galois, and we write $\Gamma := {\rm Gal}(F'/F)$. Let $a_1, \dots, a_l$ denote a $\Gamma$-stable basis of $B$-simple absolute roots for $(G,T)$. Clearly $\Gamma$ acts on the apartment $\scrA(G,T, F') \subset \scrB(G, F')$, and also on the set of affine roots $\Phi_{\rm aff} := \Phi_{{\rm aff}}(G,T, F')$ by construction, cf.~\cite[$\S1.6$]{Ti77}. We claim that there is a $\Ga$-stable set $\{\alpha _1, \dots, \alpha_l\}\subset \Phi_{{\rm aff}}$ such that the vector part of each $\alpha_j$ is $a_j$. 

Indeed, suppose we are given a $\Gamma$-orbit $\{a_{i_1}, \dots, a_{i_r}\}$ of simple roots.  We change notation and write these as $a_1, \dots, a_r$.  Choose arbitrarily an $\alpha_1 \in \Phi_{{\rm aff}}$ whose vector part is $a_{1}$.  Then for each $a_{j}$, $1 \leq j \leq r$, choose $\gamma_j \in \Gamma$ such that $\gamma_j(a_{1}) = a_{j}$, and set $\alpha_{j} := \gamma_j(\alpha_{1})$. This is well-defined because if $\gamma \in \Gamma$ fixes $a_{1}$, then it fixes $\alpha_1$ by definition of the $\Ga$-action.

Now recall from \cite[4.1.2]{BT84} that any relative root in $X^*(S)$ for $G$ is the restriction of a root in $X^*(T)$ for $G_{F'}$. Hence, any $\Gamma$-fixed point in the solution set of $\alpha_1 = \alpha_2 = \cdots = \alpha_l = 0$ is the desired absolutely special vertex of $\scrA(G,T,F')^\Gamma = \scrA(G,S,F)$. This shows existence, and also that any absolutely special vertex is special.  
\end{proof}


 Note that the above result holds even if $G$ is not tamely ramified. Now we continue with the notation and hypotheses of \S\ref{weight_free_sec}. In particular we are again assuming $G$ is adjoint, absolutely simple, and tamely ramified over $F$.

\begin{lem}\label{conjugacy_lem}
Assume that $G$ is not isomorphic to $\on{PU}_{2n+1}$ for any $n\geq 1$. Then all special vertices in $\scrB(G,F)$ are conjugate under $G(F)$, and in particular are absolutely special.
\end{lem}
\begin{proof} The last assertion follows from Lemma \ref{abs_spec_lem} using that the property of being absolutely special is invariant under $G(F)$-conjugacy. It remains to show that all special vertices are conjugate. This is implicitly contained in \cite[\S2.5]{Ti77}, and we add some details. Clearly, it is enough to show that all special points in the apartment $\scrA:=\scrA(G,T,F)$ are conjugate. Fix a special point $0\in\scrA$, and identify $\scrA=X_*(T)_{I,\bbR}$. We claim that $X_*(T)_I\subset \scrA$ is exactly the subset of special points. The claim implies the lemma because the action of $T(F)$ on $\scrA$ is via translation under $T(F)/\calT^o(\calO_F)\simeq X_*(T)_I$, and thus $T(F)$ permutes all special points. It remains to show the claim. By \cite[VI.2.2, Prop 3]{Bou}, the special points in $\scrA$ are identified with the weight lattice $P(\breve{\Sig}^\vee)$ for the \'echelonnage roots. In general, we have an inclusion $X_*(T)_I=X^*((T^\vee)^I)\subset P(\breve{\Sig}^\vee)$ which is an equality if and only if $(G^\vee)^I$ is simply connected. But this holds true by Lemma \ref{dual_list} because we excluded case ii.b) by assumption. This proves the lemma.
\end{proof}


Now assume $G=\on{PU}_{2n+1}$ for $n\geq 1$. Then up to $G(F)$-conjugation there are two kinds of special vertices. If $n\geq 2$, the local Dynkin diagram is of the form
\begin{equation}\label{C-BC_n}
\begin{aligned}
 \begin{tikzpicture}[scale=.4]
    \draw (0,0) node[anchor=east]  {$\textup{C-BC}_n$};
    \foreach \x in {2,...,5}
    \draw[xshift=\x cm,thin] (\x cm,0) circle (.1cm) node[above] {};
     \draw[xshift=1 cm,thin] (1 cm,0) circle (.1cm) node[above] {\tiny $\on{s}$};
     \draw[xshift=6 cm,thin] (6 cm,0) circle (.1cm) node[above] {\tiny $\on{as}$};
    \draw[xshift=1.15 cm,thin] (1.15 cm,.06cm) -- +(1.4 cm,0);
      \draw[xshift=1.15 cm] (1.15 cm,0) +(1.4 cm,0) node[left] {$<$} ;
    \draw[xshift=1.15 cm,thin] (1.15 cm,-.06cm) -- +(1.4 cm,0);
     \draw[xshift=2.15 cm,thin] (2.15 cm,0) -- +(1.4 cm,0);
      \draw[xshift=3.15 cm, dotted, thick] (3.15 cm,0) -- +(1.4 cm,0);
        \draw[xshift=4.15 cm,thin] (4.15 cm,0) -- +(1.4 cm,0);
     \draw[xshift=5.15 cm,thin] (5.15 cm,.06cm) -- +(1.4 cm,0);
      \draw[xshift=5.15 cm] (5.15 cm,0) +(1.4 cm,0) node[left] {$<$} ;
    \draw[xshift=5.15 cm,thin] (5.15 cm,-.06cm) -- +(1.4 cm,0); 
  \end{tikzpicture}
  \end{aligned}
\end{equation}
If $n=1$, the local Dynkin diagram is drawn in \eqref{C-BC_1} below, and looks similar. Here $F'/F$ is a ramified quadratic extension, and the vector space $V:=(F')^{2n+1}$ is equipped with a non-degenerate split Hermitian form as in \cite[3.11]{Ti77} (cf.~also \cite[1.2.1]{PR09}). The vertex labeled ``as'' is absolutely special, and corresponds to a selfdual $\calO_{F'}$-lattice in $V$ whereas the vertex labeled ``s'' is special, but not absolutely special, and corresponds to an almost modular lattice, cf.~\cite[1.2.3 a)]{PR09}. Here an $\calO_{F'}$-lattice $\La\subset V$ is called selfdual if $\La^\perp=\La$ where $(\str)^\perp$ denotes the dual  lattice with respect to the Hermitian form. The lattice is called almost modular if $u\cdot\La \subset \La^\perp$ with colength $1$ where $u\in F'$ is a uniformizer.

\subsection{Proof of Theorem \ref{sm_thm}}\label{Proof_Sm_Sec}

We start with some preliminary remarks. 
The normalization 
\begin{equation}\label{normalization_map_eq}
\tGr_{G,x}^{\leq\bar\mu}\to \Gr_{G,x}^{\leq\bar\mu}
\end{equation} 
is a finite, birational, universal homeomorphism by \cite[Prop.~3.1]{HRc}. 
In particular, the source of \eqref{normalization_map_eq} is rationally smooth if and only if its target is rationally smooth.
We therefore obtain the same list in Theorem \ref{ratl_sm_thm} for rationally smooth {\em normalized} Schubert varieties.
Also we give references below to articles which include explicit calculations for the special fibers of local models \cite{P00, Arz09, PR09, HPR}. To apply these references we need to often replace the adjoint group $G$ by a suitable central extension $\tilde{G}\to G$ such that $\pi_1(\tilde{G}_\der)=0$, cf.~\cite[(2.11)]{HPR}. 
Then the Schubert varieties for $\tilde G$ map isomorphically onto the normalized Schubert varieties for $G$, cf.~Proposition \ref{Ad_Iso_prop} using the normality of Schubert varieties for $\tilde G$ \cite[Thm.~6.1]{PR08}.
Hence, by \cite[Thm 9.1]{PZ13} the normalized Schubert variety is isomorphic to the special fiber of a suitable local model for $\tilde G$ which allows us to use these references.
Also we give references below to articles which contain results about the singularity of Schubert varieties \cite{EM99, MOV05} over the complex numbers. 
Here we refer to $\S\ref{Reduction_Sec}$ below for the reduction to $k = \mathbb C$ which allows us to use these references.

If $\bar{\mu}$ is minuscule, or if $G$ is an odd-dimensional ramified unitary group, $x$ is not absolutely special and $\bar{\mu}$ is quasi-minuscule, then $\tGr_{G,x}^{\leq \bar{\mu}}$ is smooth. Indeed, if $\bar{\mu}$ is minuscule, then $\Gr_{G,x}^{\leq \bar{\mu}}=\Gr_{G,x}^{\bar{\mu}}$ is a single orbit, and hence is smooth, so that its normalization is smooth as well.
The other case was observed by the second named author, and follows from an explicit calculation, cf.~\cite[Prop 4.16]{Arz09}.

Conversely assume that $\tGr_{G,x}^{\leq \bar{\mu}}$ is smooth. Then it is rationally smooth as well, and hence appears in the list of Theorem \ref{ratl_sm_thm}. We need to exclude from that list all the Schubert varieties which are singular. If $\bar{\mu}$ is minuscule, then $\tGr_{G,x}^{\leq \bar{\mu}}$ is smooth as argued above. Therefore, we have reduced to the case where $\bar{\mu}$ is not minuscule. In what follows, we list groups according to the type of the dual group $(G^\vee)^I$. 
\medskip

\noindent \textbf{Type $A_1$}, and $\bar{\mu}$ arbitrary: Note that there is an $l\geq 2$ such that $\bar{\mu}=l\cdot \om_1$. If $G=\PGL_2$ is split, then $\tGr_{G,x}^{\leq \bar{\mu}}$ is singular by \cite[\S5.1]{MOV05} (using $\S\ref{Reduction_Sec}$ below to reduce to $k = \mathbb C$ here, and below), cf.~also \cite[Thm 9.2]{Mue08} for an explicit matrix calculation. 
These cases are therefore excluded. If $G$ is not split, then according to Lemma \ref{dual_list} it is the $3$-dimensional quasi-split ramified projective unitary group. Note that only the weights $\bar{\mu}=l\cdot \om_1$ for {\em even} $l\geq 2$ appear in this case because $(G^\vee)^I=\on{SO}_3\simeq \PGL_2$ 
is not simply connected. By \S\ref{Reduction_Sec} and Proposition \ref{3dim_unitary} below, the normalized Schubert variety $\tGr_{G,x}^{\leq \bar{\mu}}$ is smooth only in the case where $x$ is special but not absolutely special, and $\bar{\mu}=2\om_1$ is quasi-minuscule, and so only this case is not excluded.\smallskip\\
\textbf{Type $A_n$}, $n\geq 2$ and $\bar{\mu}=l\cdot \om_i$, $i\in\{1,n\}$, $l\geq 2$: By Lemma \ref{dual_list}, the group $G$ is split, and hence $G=\PGL_{n+1}$. The singularity of these normalized Schubert varieties is a particular case of \cite{EM99, MOV05} (again reduce to $k = \mathbb C$). Let us be more specific. 
The inverse transpose morphism $G\to G$, $g\mapsto (g^{-1})^t$ induces an isomorphism on affine Grassmannians flipping the connected components, and in particular restricts to an isomorphism $\tGr_{G,x}^{\leq l\cdot\om_1}\simeq \tGr_{G,x}^{\leq l\cdot\om_{n}}$ for all $l\geq 2$. We are thus reduced to the case where $\bar{\mu}=l\cdot \om_1$, $l\geq 2$. Also by our general remarks above, we can identify the normalized Schubert variety with an ordinary Schubert variety in $\GL_{n+1}$. 
We can therefore assume that $G=\GL_{n+1}$ and $\tGr_{G,x}^{\leq l\cdot\om_1}=\Gr_{G,x}^{\leq l\cdot\om_1}$.

In this case, we consider the element $\bar{\la}=(l-1)\cdot \om_1+\om_2$. Then $\bar{\mu}-\bar{\la}$ is a simple coroot. By the Levi lemma of \cite[3.4]{MOV05}, the boundary of $\Gr_{G,x}^{\leq \bar{\la}}\subset \Gr_{G,x}^{\leq \bar{\mu}}$ is smoothly equivalent to the boundary of a Schubert variety for $\GL_2$, and hence is singular by the type $A_1$ case above. Therefore all the cases in this paragraph are excluded.

\noindent\textbf{Type $B_n$}, $n\geq 2$, and $\bar{\mu}$ quasi-minuscule: If $G$ is split, then $\tGr_{G,x}^{\leq \bar{\mu}}$ is singular, cf.~\cite[\S 2.9, 2.10]{MOV05}. If $G$ is not split, then according to Lemma \ref{dual_list} we are left with the cases ii.b), ii.c) for any $n\geq 2$. In case ii.b), the group $G$ is a quasi-split unitary group on an $2n+1$-dimensional Hermitian space. If $x$ is absolutely special, then $\tGr_{G,x}^{\leq \bar{\mu}}$ is singular at the base point by \cite[Thm 4.5, Lem 4.7]{P00} (cf.~also \cite[\S9, 3.b)]{HPR}). This case is therefore excluded. If $x$ is special, but not absolutely special, then $\tGr_{G,x}^{\leq \bar{\mu}}$ is smooth at the base point by \cite[Prop 4.16]{Arz09} (cf.~also \cite[\S9, 3.a)]{HPR}). This case is therefore not excluded.

In case ii.c), the group $G$ is a ramified orthogonal group on a $2n+2$-dimensional space. The normalized Schubert variety $\tGr_{G,x}^{\leq \bar{\mu}}$ is singular by \cite[\S9, 3.c)]{HPR}. Note that by Lemma \ref{conjugacy_lem} all special vertices are conjugate under $G(F)$ so that we only need to consider the choice of $x$ that is handled in \textit{loc.~cit.}. This case is excluded.\smallskip\\
\textbf{Type $C_3$}, and $\bar{\mu}=\omega_3$: If $G$ is split, then $G=\on{PSp}_6$. We have $\bar{\mu}=\om_3>\om_1=:\bar{\la}$, and $\bar{\la}$ is minuscule. Checking the tables in \cite{Bou}, we see that $\bar{\la}$ is equal to zero on the root subsystem $\on{supp}(\bar{\mu}-\bar{\la})= \{\al^\vee_2,\al^\vee_3\}$, viewing the latter as simple coroots in $\breve{\Sigma}^\vee$. 
Also by our general remarks above, we can identify the normalized Schubert variety with an ordinary Schubert variety in $\on{GSp}_{6}$. 
We can therefore assume that $G=\on{GSp}_{6}$ and $\tGr_{G,x}^{\leq \bar\mu}=\Gr_{G,x}^{\leq \bar\mu}$. By the Levi lemma of \cite[3.4]{MOV05}, the boundary of $\Gr_{G,x}^{\leq \bar{\la}}\subset \Gr_{G,x}^{\leq \bar{\mu}}$ is smoothly equivalent to the quasi-minuscule singularity of type $C_2=B_2$, which is singular by the previous case. This case is excluded. 

If $G$ is not split, then according to Lemma \ref{dual_list} we are in case ii.a) for $n=3$, i.e., $G$ is a quasi-split ramified unitary group on a $6$-dimensional Hermitian space. The normalized Schubert variety $\tGr_{G,x}^{\leq \bar{\mu}}$ is isomorphic to the special fiber of the local model of the unitary similitudes group for signature $(3,3)$, and this is singular by \cite[\S9, 2)]{HPR} (cf.~the equations given in \cite[(5.6)]{PR09}). This case is excluded. \smallskip\\
\textbf{Type $G_2$}, and $\bar{\mu}$ quasi-minuscule: If $G$ is split and $\mu$ is quasi-minuscule, then \cite[2.9]{MOV05} gives a conceptual proof showing that the base point $e$ in $\tGr_{G,x}^{\leq \bar{\mu}}$ is singular (this does not use the Kumar criterion). If $G$ is not split, then $G$ is the ``ramified triality'', and the base point is again singular. This is the most difficult case in our classification, and it is treated in \S\ref{Triality_Sec} below, cf.~Theorem \ref{Sing_Thm}. In both the split and non-split cases, we are using \S\ref{Reduction_Sec} for the reduction to the case $k=\bbC$. Thus, the normalized Schubert variety $\tGr_{G,x}^{\leq \bar{\mu}}$ is singular at the base point, and hence it is excluded. This finishes the proof of the classification, and hence the proof of Theorem \ref{sm_thm}. 

\subsection{A conjecture on minimal degenerations}\label{Conjecture_Sec} Theorem \ref{sm_thm} classifies the normalized Schubert varieties in twisted affine Grassmannians which are smooth. If the group $G$ is split and ${\rm char}(k)=0$, the result proven in \cite{EM99, MOV05} is stronger. In this case, every (normalized) Schubert variety is singular along its boundary, i.e., $\Gr_{G,x}^{\bar{\mu}} = \tGr_{G,x}^{\bar{\mu}}$ is exactly the smooth locus in $\Gr_{G,x}^{\leq \bar{\mu}} = \tGr_{G,x}^{\leq \bar{\mu}}$. As the phenomenon of exotic smoothness shows, this fails in twisted affine Grassmannians for general special vertices.

\begin{conject}\label{conject}
If $x$ is absolutely special, then the smooth locus of $\tGr_{G,x}^{\leq \bar{\mu}}$ is precisely 
$\tGr_{G,x}^{\bar{\mu}}$.
\end{conject}

In Corollary \ref{conjecture.evidence.lem} below we give some evidence for this conjecture.
The essential difficulty in proving the conjecture consists in handling absolutely simple, non-split groups over $\bbC\rpot{t}$.

\begin{rmk}
i) If Conjecture \ref{conject} holds for the normalized Schubert varieties, then the same is true for the non-normalized Schubert varieties as well. Indeed, the finite birational universal homeomorphism $\tGr_{G,x}^{\leq \bar{\mu}} \rightarrow  \Gr_{G,x}^{\leq \bar{\mu}}$ induces an isomorphism over the smooth locus of the target. \smallskip\\
ii) It would also be interesting to determine the type of singularities which arise. 
The calculations in \S\ref{Triality_Sec} indicate that these might be different from the minimal degeneration singularities for split groups.
\end{rmk}

\section{Reduction of the remaining cases to $k=\bbC$}\label{Reduction_Sec}
In order to treat the remaining cases in \S\ref{sec_unitary} and \S\ref{Triality_Sec} below, we first reduce the proof that the normalized Schubert varieties in question are singular to the case where $k=\bbC$. Let $W:=W(k)$ be the ring of Witt vectors of $k$ equipped with the natural map $W\to k$. Let $K=\on{Frac}(W)$ be the field of fractions. As the group $G$ is tamely ramified, the twisted affine Grassmannian together with the Schubert varieties lift to $W$, cf.~\cite[\S 7]{PR08}. 

More precisely, there exists a smooth, affine group scheme with connected fibers $\underline \calG\to W\pot{u}$ whose base change $\underline G$ to $W\rpot{u}$ is reductive, and whose base change to $\kappa\pot{u}$ for $\kappa=k, K$ is the parahoric group scheme for $\underline G_{\kappa\rpot{u}}$ attached with the ``same'' facet, cf.~\cite[Cor.~4.2 (2)]{PZ13}. 
We note that $\underline \calG/W\pot{u}$ is a special case of the ``parahoric'' group schemes constructed in \cite[\S4]{PZ13}.

Hence, as in \cite[\S 7]{PR08} there exists a twisted affine Grassmannian $\underline{\Gr}_{G,x}$ defined over $W$, and for every $\bar{\mu}\in X_*(T)_I^+$ a normalized Schubert variety $\underline{\tGr}^{\leq \bar{\mu}}_{G,x}$ such that 
\begin{equation}\label{Lift}
\underline{\tGr}^{\leq \bar{\mu}}_{G,x}\otimes k= \tGr^{\leq \bar{\mu}}_{G,x}.
\end{equation}
Note that we are using the identification $X_*(T)_I=X_*(\underline{T}_K)_{I_K}$ which is compatible with the dominance order where $\underline{T}$ is the lift of $T$ over $W$: this is an immediate consequence of the identification of apartments for $\underline G_{K\rpot{u}}$ and $\underline G_{k\rpot{u}}=G$ \cite[(4.2)]{PZ13}. Further, we use the fact that formation of $\underline{\tilde{\Gr}}_{G,x}^{\leq \bar{\mu}}$ commutes with base change. 
For simply connected groups this is proved in \cite[Prop.\,9.11 (a) and 9.g]{PR08}, and the case of adjoint groups  follows from this is by using a standard reduction to affine flag varieties and a translation to the neutral component as e.g.~in \cite[Prop.~3.1]{HRc}.
The affine Grassmannian in the generic fiber $\underline{\Gr}_{G,x}\otimes K$ is of the same type as the affine Grassmannian in the special fiber $\Gr_{G,x}$, so that $\underline{\tGr}_{G,x}^{\leq \bar{\mu}}\otimes K$ is the Schubert variety in $\underline{\Gr}^{\leq \bar{\mu}}_{G,x}\otimes K$ for the same $\bar{\mu}\in X_*(T)_I$ (using the normality of Schubert varieties in characteristic $0$). 

As the singular locus in $\underline{\tGr}_{G,x}^{\leq\bar{\mu}}$ is closed, the generic fiber $\underline{\tGr}_{G,x}^{\leq\bar{\mu}}\otimes K$ being singular implies the special fiber $\Gr_{G,x}^{\leq\bar{\mu}}$ is singular. 
We are therefore reduced to the case that $k$ is an algebraically closed field of characteristic zero. 
Furthermore, every such group is already defined over $\bar\bbQ\rpot{t}$ so that we further reduce to the case $k=\bbC$. 
We will assume this whenever convenient in what follows.

With a view towards Conjecture \ref{conject}, we note that if the smooth locus of the generic fibre $\underline{\tGr}_{G,x}^{\leq\bar{\mu}}\otimes K$ is precisely $\underline{\tGr}_{G,x}^{\bar{\mu}}\otimes K$, then the smooth locus of the special fibre $\underline{\tGr}_{G,x}^{\leq\bar{\mu}}\otimes k=\tGr_{G,x}^{\leq\bar{\mu}}$ is precisely $\tGr_{G,x}^{\bar{\mu}}$. Indeed, the formation of $\underline{\tGr}_{G,x}^{\bar{\lambda}}$ commutes with base change; and hence for a maximal element $\bar{\lambda} < \bar{\mu}$, any point in the $\bar{\lambda}$-stratum of $\underline{\tGr}_{G,x}^{\leq \bar{\mu}} \otimes k$ is a specialization of a point in the $\bar{\lambda}$-stratum of $\underline{\tGr}_{G,x}^{\leq \bar{\mu}} \otimes K$.
Invoking \cite{EM99, MOV05} we thus have proven:

\begin{cor}\label{conjecture.evidence.lem} 
i\textup{)} If Conjecture \ref{conject} holds for $k=\bbC$, then it holds for general fields $k$.\smallskip\\
ii\textup{)} If $G$ is split, then Conjecture \ref{conject} holds.
\end{cor}
\qed

\section{The three-dimensional quasi-split ramified unitary groups} \label{sec_unitary}
Let $k$ be an algebraically closed field with $\on{char}(k) \neq 2$. Let $F=k\rpot{t}$, and let $F'/F$ be a quadratic ramified extension.  
Let $G=\on{PU}_3$, and let $x\in \scrB(G,F)$ be a special vertex. 
Up to conjugation by $G_\ad(F)$, there are two kinds of vertices: one is absolutely special, and one is special, but not absolutely special. The local Dynkin diagram of $G$ is:
\begin{equation}\label{C-BC_1}
\begin{aligned}
 \begin{tikzpicture}[scale=.4]
    \draw (0,0) node[anchor=east]  {$\textup{C-BC}_1$};
    \draw[xshift=1 cm,thin] (1 cm,0) circle (.1cm) node[above] {\tiny s} node[below] {\tiny 1};
    \draw[xshift=2 cm,thin] (2 cm,0) circle (.1cm) node[above] {\tiny as} node[below] {\tiny 0};
    \draw[xshift=1.15 cm, thin] (1.15 cm,.1cm) -- +(1.4 cm,0);
        \draw[xshift=1.15 cm, thin] (1.15 cm,.03cm) -- +(1.4 cm,0);
                        \draw[xshift=1.15 cm] (1.15 cm,0) +(1.4 cm,0) node[left] {$<$};
                \draw[xshift=1.15 cm, thin] (1.15 cm,-.03cm) -- +(1.4 cm,0);
        \draw[xshift=1.15 cm, thin] (1.15 cm,-.1cm) -- +(1.4 cm,0);
  \end{tikzpicture}
\end{aligned}  
\end{equation}
Note that $\Gr_{G,x}$ is connected with a linear order relation on the Schubert varieties, i.e., there are no minuscule Schubert varieties. Recall that the pair $(\bar{\mu},x)$ is called of exotic smoothness if $\bar{\mu}$ is quasi-minuscule, and $x$ is special, but not absolutely special, i.e., the vertex labeled $1$ in \eqref{C-BC_1}. The aim of this section is to prove the following proposition.

\begin{prop}\label{3dim_unitary}
Assume that $(\bar{\mu},x)$ is not of exotic smoothness. Then the smooth locus of $\tGr_{G,x}^{\leq \bar{\mu}}$ is exactly $\tGr_{G,x}^{\bar{\mu}}$. In particular, $\tGr_{G,x}^{\leq \bar{\mu}}$ is singular if $\bar{\mu}$ is non-trivial.  
\end{prop}

We may pass to working with $G = {\rm SU}_3$ thanks to Proposition \ref{Ad_Iso_prop}. 
Also, we use \S\ref{Reduction_Sec} to reduce the proof of Proposition \ref{3dim_unitary} to case where $k=\bbC$.
In this case, $\tGr_{G,x}^{\leq\bar{\mu}}=\Gr_{G,x}^{\leq\bar{\mu}}$ is the ordinary Schubert variety.  
By \cite[\S9.f]{PR08} the twisted affine flag varieties for simply connected groups agree with the Kac-Moody affine flag varieties so that Kumar's criterion \cite{Ku96} is applicable. This is a criterion for smoothness of Schubert varieties in terms of affine Weyl group combinatorics. We first reduce the proof of Proposition \ref{3dim_unitary} to Schubert varieties in the affine flag variety in \S\ref{Aff_Flag}, and then recall Kumar's criterion in \S\ref{Kumar_criterion} below. The final verifications are made in \S\ref{Proof_Kumar} below.

\subsection{Preliminaries on Schubert varieties}
 We write $F'=k\rpot{u}$ for a choice of uniformizer $u$ with $u^2=t$, and we fix a basis giving an isomorphism $V=(F')^3$ such that the Hermitian form is given by the anti-diagonal matrix $\on{antidiag}(1,1,1)$. With respect to this basis, we let $T\subset G$ be the diagonal torus, and $B\subset G$ the Borel subgroup of upper triangular matrices. 

We define the $\calO_{F'}$-lattices $\La_0:=\calO_{F'}^3$ and $\La_1:=u^{-1}\calO_{F'}\oplus \calO_{F'}^2$. Up to conjugation by $\on{SU}_3(F)$, the vertex $x$ corresponds either to the absolutely special vertex given by the selfdual lattice $\La_0$, or to the special, but not absolutely special vertex given by the almost modular lattice $\La_1$, cf.~\eqref{C-BC_n}. We fix the base alcove $\bba$ which corresponds to the Iwahori subgroup in $\on{SU}_3(F)$ given by the stabilizer of the lattice chain $\La_0\subset \La_1$.

Observe that $(G^\vee)^I$ is an adjoint group (cf.\,Proposition \ref{fix_pt_prop}) and so $X_*(T)_I = X^*((T^\vee)^I)$ is generated by $\breve{\Sigma}^\vee$.  Write $e_1= \bar{\mu}_1$ for the simple \'{e}chelonnage coroot. In this way we identify $X_*(T)_I^+=\bbZ_{\geq 0}$, and we denote by $\bar{\mu}_{l}\in X_*(T)_I^+$ the element which corresponds to $l \in \bbZ_{\geq 0}$. Explicitly, $\bar{\mu}_{l}$ is under the Kottwitz map given by the class of the diagonal matrix $\on{diag}(u^l,1,(-u)^{-l})\in T(F)$. As closed subschemes in the affine Grassmannian we have
\[
\{e\}=\Gr_{G,x}^{\leq \bar{\mu}_0}\subset \Gr_{G,x}^{\leq \bar{\mu}_1}\subset \Gr_{G,x}^{\leq \bar{\mu}_2}\subset\ldots,
\]
and $\Gr_{G,x}^{\leq \bar{\mu}_{l}}\bslash \Gr_{G,x}^{\leq \bar{\mu}_{l-1}}=\Gr_{G,x}^{\bar{\mu}_{l}}$ which is of dimension $2l$ (because the base alcove corresponds to the interval $(0,\frac{1}{2})$, see below). The element $\bar{\mu}_1$ is the unique quasi-minuscule element.

\subsection{Reduction to the affine flag variety}\label{Aff_Flag}
We consider the Iwahori $\calO_F$-group scheme $\calG_\bba$ given by the automorphisms of the lattice chain $\La_0\subset \La_1$, and denote by $\Fl:=LG/L^+\calG_\bba$ the associated twisted affine flag variety in the sense of \cite[1.c]{PR08}. Let $W=W(G,T,F)$ be the affine Weyl group. For each $w\in W$, we denote by $\Fl^{\leq w}\subset \Fl$ the $L^+\calG_\bba$-Schubert variety associated with the base point $n_w\in \Fl(k)$ corresponding to $w$, cf.~\cite[\S8]{PR08}.  

The canonical projection $\pi\co \Fl\,\to\, \Gr_{G,x}$ is representable by a smooth proper surjective morphism of relative dimension $1$, cf.~\cite[Lem 4.9 i)]{HRa}. Thus, for each $l\geq 1$ there exists a unique element $w_{l,x}\in W$ such that as subschemes of the affine flag variety
\[
\Fl^{\leq w_{l,x}}\,=\,\pi^{-1}(\Gr_{G,x}^{\leq \bar{\mu}_l}).
\]
Since the projection $\Fl^{\leq w_{l,x}}\to \Gr_{G,x}^{\leq \bar{\mu}_l}$ is smooth, to show $\on{Gr}^{\leq \bar{\mu}_l}_{G, x}$ is singular it is enough to show that $\Fl^{\leq w_{l,x}}$ is singular at a point $v_{l,x}$ lying over $\bar{\mu}_{l-1}$. We need to explicate the elements $w_{l,x}$ in terms of the affine Weyl group $W$, and need to make suitable choices for $v_{l,x}$.

\subsubsection{Affine Roots} We have the perfect pairing $\lan\str,\str\ran\co X^*(T)^I_\bbR\times X_*(T)_{I,\bbR}\to \bbR$ of $1$-dimensional $\bbR$-vector spaces. Let $\epsilon_1 \in X^*(T)^I_{\mathbb R}$ be such that $\langle \epsilon_1, e_1\rangle = 1$. The set of affine roots $\Phi_\aff=\Phi_\aff(G,T,F)$ is given by 
\[
\Phi_\aff\,=\,\{\pm\epsilon_1+\bbZ;\; \pm 2\epsilon_1+\bbZ\}.
\]
It follows that the simple affine roots are $\al_1=\epsilon_1$, $\al_0=-2\epsilon_1+1$, and the simple \'{e}chelonnage root is $\alpha_{\rm ech} = 2\epsilon_1$. These roots have coroots $\al_1^\vee=2e_1$, $\al_0^\vee=-e_1+{1\over 2}$, and $\alpha_{\rm ech}^\vee = e_1 = \bar{\mu}_1$.  Note this is consistent with our description of $\bar{\mu}_1$ above. The base alcove $\bba$ is the open interval $(0,{1\over 2})\subset \bbR$.  In this notation, we have $\bar{\mu}_l = le_1 = l\alpha^\vee_{\rm ech}$.

\subsubsection{Simple reflections} The affine Weyl group $W$ has a Coxeter group structure given by the choice of the base alcove $\bba$. We denote by $\leq$ the partial order, and by $\ell(w)\in\bbZ_{\geq 0}$ the length of an element $w\in W$. We let $s_0:=s_{\al_0}$ be the simple affine reflection given by $\al_0$, and we let $s_1:=s_{\al_1}$ be the simple reflection given by $\al_1$. We have the group presentation
\[
W\,=\, \lan s_0,s_1\;|\;s_0^2=s_1^2=1\ran. 
\]
The group $W$ acts on $\Phi_\aff$. We have $s_i(\al_i)=-\al_i$ for $i=0,1$, and 
\begin{align*}
s_0(\al_1)\,&=\, \al_1-\lan\al_1,\al_0^\vee\ran\al_0\,=\,\al_1+\al_0,\\
s_1(\al_0)\,&=\, \al_0- \lan\al_0,\al_1^\vee\ran\al_1\,=\,\al_0+4\al_1.
\end{align*}
Translation by $\bar{\mu}$ takes the base alcove ${\bf a} = (0, {1\over 2})$ to the interval $(1, {3 \over 2})$. Therefore, for $l \geq 1$, translation by $\bar{\mu}_l$ is an element in $W$ with reduced expression $(s_0s_1)^l$. In what follows we will abbreviate such expressions.

\subsubsection{Cases}\label{Cases_sec} We extend the definition to every $i\in\bbZ$ by $s_i:=s_0$ (resp.~$\al_i:=\al_0$) if $i$ is even, and $s_i:=s_1$ (resp.~$\al_i:=\al_1$) if $i$ is odd. Further, for every pair $i,j\in \bbZ$, we define $s_{i,j}:=s_is_{i+1}\cdots s_j$ if $i\leq j$ and $s_{i,j}:=1$ if $i>j$. Fix $l\geq 1$ and consider the following two cases.\smallskip\\
\textbf{Case A.} Let $x$ be the absolutely special vertex given by the selfdual lattice $\La_0$. Then $\calG_x(\calO_F)$ contains the affine root groups given by $\pm\al_1$. In this case, we have $w_{l,x}=s_{1,2l+1}$, and we fix this reduced expression. We define $v_{l,x}:=s_{1,2l-1}$. \smallskip\\
\textbf{Case B.} Let $x$ be special, but not absolutely special vertex given by the almost modular lattice $\La_1$. Then $\calG_x(\calO_F)$ contains the affine root groups given by $\pm\al_0$. In this case, we have $w_{l,x}=s_{0,2l}$, and we fix this reduced expression. We define $v_{l,x}:=s_{0,2l-2}$, i.e., the roles of $0$ and $1$ are interchanged.

\subsection{Kumar's criterion}\label{Kumar_criterion} From now on assume that $k=\bbC$. Then $\Fl$ is the Kac-Moody affine flag variety (cf.~\cite{Ku96}) associated with the generalized Cartan matrix of rank two given by
\[
\begin{pmatrix} \lan\al_0,\al_0^\vee\ran &  \lan\al_0,\al_1^\vee\ran\\
 \lan\al_1,\al_0^\vee\ran &  \lan\al_1,\al_1^\vee\ran
\end{pmatrix}
= 
\begin{pmatrix} 2 & -4\\
-1 & 2
\end{pmatrix}.
\] 

Let $Q$ be the fraction field of the symmetric algebra of the root lattice. Let $w\in W$, and fix a reduced decomposition $w=s_{1}\cdot\ldots\cdot s_{n}$. For $v\leq w$, we define
\begin{equation}\label{Kumar_def}
e_vX(w)\defined \sum\prod_{i=1}^n\tilde{s}_1\cdots \tilde{s}_i (\al_i)^{-1}\,\in\, Q,
\end{equation}
where the sum runs over all sequences $(\tilde{s}_1,\ldots,\tilde{s}_n)$ such that either $\tilde{s}_i=1$ or $\tilde{s}_i=s_i$ for every $i$, and $\tilde{s}_1\cdot \ldots\cdot \tilde{s}_n=v$. The following theorem is \cite[Thm 5.5 (b); Thm 8.9]{Ku96}.

\begin{thm}\label{Kumar_thm}
The Schubert variety $\Fl^{\leq w}$ is smooth at $v$ if and only if 
\begin{equation}\label{Kumar_eq}
e_vX(w)\,=\, (-1)^{\ell(v)}\prod_{\al\in \Phi^+_\aff;\, s_\al v\leq w}\al^{-1},
\end{equation}
where $\Phi^+_\aff\subset \Phi_\aff$ is the set of positive affine roots, and $s_\al$ denotes the associated reflection.
\end{thm}
\qed

\subsection{End of the proof}\label{Proof_Kumar} By Theorem \ref{Kumar_thm}, we need to calculate the expression $e_{v_{l,x}}X(w_{l,x})$ in both cases A and B. Note that there are $2l$ subexpressions of $v_{l,x}$ in $w_{l,x}$ defined by deleting two neighboring simple reflections, and that all subexpressions are of this form. We use the notation introduced in \S\ref{Cases_sec}.\smallskip\\
\textit{Case A:} An elementary calculation gives
\begin{equation}\label{Abs_Spec_Calc_A}
e_{v_{l,x}}X(w_{l,x})=\left(\prod_{i=1}^{2l-1}s_{1,i}(\al_i)^{-1}\right)\cdot \underbrace{\left(\sum_{i=0}^{2l-1}s_{1,i}(\al_0)^{-1}s_{1,i}(\al_1)^{-1}\right)}_{=:A_{l}}.
\end{equation}
\textit{Case B:} The same calculation as before gives
\begin{equation}\label{Abs_Spec_Calc_B}
e_{v_{l,x}}X(w_{l,x})=\left(\prod_{i=0}^{2l-2}s_{0,i}(\al_i)^{-1}\right)\cdot \underbrace{\left(\sum_{i=0}^{2l-1}s_{0,i-1}(\al_0)^{-1}s_{0,i-1}(\al_1)^{-1} \right)}_{=:B_{l}}.
\end{equation}

\begin{cor}\label{red_cor_sm}
The Schubert variety $\Fl^{\leq w_{l,x}}$ is smooth at $v_{l,x}$ if and only if 
$$
A_l=-\al_0^{-1}s_{1,2l}(\al_0)^{-1} = \al_0^{-1} s_{1,2l-1}(\al_0)^{-1} \hspace{.15in} {\mbox resp.,}\hspace{.15in}B_l= -\al_{1}^{-1}s_{0,2l-1}(\al_1)^{-1} = \al_1^{-1}s_{0,2l-2}(\al_1)^{-1}. 
$$
\end{cor}

\begin{proof} The right hand side in \eqref{Kumar_eq} takes the form
\[
\text{Case A:}\;\;-\prod_{i=0}^{2l}s_{1,i-1}(\al_i)^{-1}=-\prod_{i=0}^{2l}s_{1,i}(\al_i)^{-1}\;\;\;\;\;\;\;\;\;\;\text{Case B:}\;\; -\prod_{i=-1}^{2l-1}s_{0,i-1}(\al_i)^{-1}=-\prod_{i=-1}^{2l-1}s_{0,i}(\al_i)^{-1}.
\]
Here we used $s_i(\al_i)=-\al_i$ and the conventions $s_{1,-2}=s_{1,-1}=s_{1,0}=1$ to keep track of the signs. By comparing this with \eqref{Abs_Spec_Calc_A} (resp.~\eqref{Abs_Spec_Calc_B}), the corollary follows from Theorem \ref{Kumar_thm}. 
\end{proof}

We now need to calculate $A_l$ and $B_l$ for every $l\geq 1$. The following identities are useful.

\begin{lem}\label{useful_lem} For all $i\in \bbZ_{\geq 0}$, one has \smallskip\\
i\textup{)} $\al_0+s_1(\al_0)=2$, and $s_{1,2i}(\al_0)=\al_0-2i$, and $s_{1,2i+1}(\al_0)=s_1(\al_0)+2i$;\smallskip\\
ii\textup{)} $\al_1+s_0(\al_1)=1$, and $s_{0,2i-1}(\al_1)=\al_1-i$, and $s_{0,2i}(\al_1)= s_0(\al_1)+i$. 
\end{lem}
\begin{proof} We have $s_1(\al_0)=\al_0+4\al_1=2-\al_0$, and $s_0(\al_1)=\al_0+\al_1=1-\al_1$. The remaining identities are proved by an easy induction, and are left to the reader.  
\end{proof}

\noindent\textit{Case A:} The number of summands in $A_l$ is even, and we add terms in consecutive pairs as follows. For every $0\leq i\leq l-1$, one has
\[
s_{1,2i}(\al_0)^{-1}s_{1,2i}(\al_1)^{-1}+ s_{1,2i+1}(\al_0)^{-1}s_{1,2i+1}(\al_1)^{-1} \,=\, 4\cdot s_{1,2i}(\al_0)^{-1}s_{1,2i+1}(\al_0)^{-1},
\]
where we used $s_{1,2i+1}(\al_1)=-s_{1,2i}(\al_1)$ and $s_{1,2i+1}(\al_0)=s_{1,2i}(\al_0)+4s_{1,2i}(\al_1)$. This shows
\[
A_l\,=\, 4\cdot \sum_{i=0}^{l-1}s_{1,2i}(\al_0)^{-1}s_{1,2i+1}(\al_0)^{-1}\,=\, 4l\cdot \al_0^{-1}s_{1,2l-1}(\al_0)^{-1}.
\]
For the last equality we use Lemma \ref{telescope_lem} below applied to $\al_0+s_1(\al_0)=2$ which is justified by Lemma \ref{useful_lem} i). Hence, by Corollary \ref{red_cor_sm} the Schubert variety $\Fl^{\leq w_{l,x}}$ is singular at $v_{l,x}$ for all $l\geq 1$ in this case.\smallskip\\
\textit{Case B:} Again the number of summands in $B_l$ is even, and we add terms in consecutive pairs. For every $0\leq i\leq l-1$, one has
\[
s_{0,2i-1}(\al_0)^{-1}s_{0,2i-1}(\al_1)^{-1}+ s_{0,2i}(\al_0)^{-1}s_{0,2i}(\al_1)^{-1} \,=\, s_{0,2i-1}(\al_1)^{-1}s_{0,2i}(\al_1)^{-1},
\]
where we used $s_{0,2i}(\al_0)=-s_{0,2i-1}(\al_0)$ and $s_{0,2i}(\al_1)=s_{0,2i-1}(\al_0)+s_{0,2i-1}(\al_1)$. This shows
\[
B_l\,=\, \sum_{i=0}^{l-1}s_{0,2i-1}(\al_1)^{-1}s_{0,2i}(\al_1)^{-1}\,=\, l\cdot \al_1^{-1}s_{0,2l-2}(\al_1)^{-1}.
\]
For the last equality we use Lemma \ref{telescope_lem} below applied to $\al_1+s_0(\al_1)=1$ which is justified by Lemma \ref{useful_lem} ii). Hence, by Corollary \ref{red_cor_sm} the Schubert variety $\Fl^{\leq w_{l,x}}$ is singular (resp.~smooth) at $v_{l,x}$ for $l\geq 2$ (resp.~$l=1$). This finishes the proof of Proposition \ref{3dim_unitary}.

\begin{lem}\label{telescope_lem} 
Let $n\in \bbZ_{\geq 1}$, and $\al,\be\in Q\backslash n\bbZ$ \textup{(}e.g.,\,$\alpha, \beta$ are affine roots\textup{)} with $\al+\be=n$. For any $l\geq 1$, one has
\[
\sum_{i=0}^{l-1}{1\over {(\al-ni)(\be+ni)}} = {l\over {\al(\be+n(l-1))}}.
\]
\end{lem}
\begin{proof} This is elementary, and left to the reader.
\end{proof}

\begin{rmk} The form of $A_l$ and $B_l$ also shows that the Schubert variety $\Fl^{\leq w_{l,x}}$ is rationally smooth at $v_{l,x}$, cf.~\cite[Thm 5.5 (a); Thm 8.9]{Ku96}. This is in accordance with Theorem \ref{ratl_sm_thm}.
\end{rmk}

\section{The quasi-minuscule Schubert variety for the ramified triality}\label{Triality_Sec}
Let $k$ be an algebraically closed field of characteristic $\not = 3$, and let $F=k\rpot{t}$ be the Laurent series local field. Let $G$ be the twisted triality over $F$, i.e., the up to isomorphism unique quasi-split but non-split form of $\on{Spin}_8$ over $F$. Note that $G$ splits over a totally ramified extension of $F$ of degree $3$, and is therefore tamely ramified by the assumption on $k$. Let $x\in \scrB(G,F)$ be a special vertex in the Bruhat-Tits building, and denote by $\calG_x$ the associated special parahoric group scheme over $\calO_F=k\pot{t}$. By \cite[\S2.5]{Ti77} the group $G_\ad(F)$ acts transitively on the set of special vertices. Therefore we may justifiably denote by $\Gr_{G}:=LG/L^+\calG_x$ the twisted affine Grassmannian in the sense of \cite{PR08}. 
We also note that all Schubert varieties inside $\Gr_G$ are normal by \cite[Thm.~6.1]{PR08}.

\begin{thm}\label{Sing_Thm}
The quasi-minuscule Schubert variety in $\Gr_{G}$ is a $6$-dimensional projective $k$-variety which is rationally smooth, but singular at the base point. 
\end{thm}

Rational smoothness follows from Theorem \ref{ratl_sm_thm}, cf.~also Remark \ref{Ratl_Smooth_Rmk} below. For the proof that the quasi-minuscule Schubert variety is singular at the base point, we use \S\ref{Reduction_Sec} to reduce to the case $k=\bbC$ first which we assume in \S\ref{Nil_Orbits}, \S\ref{End_Proof} below. Here our method is similar to the method in \cite[2.9-2.10]{MOV05}: we construct a neighborhood of the base point inside the quasi-minuscule Schubert variety in a certain space of nilpotent matrices. We show that the tangent space of our Schubert variety at the base point is $7$-dimensional, and hence this point is singular. Note that it is quite different from the quasi-minuscule Schubert variety for a split group of type $G_2$ in which case the tangent space at the identity is of dimension $14$ (=dimension of the Lie algebra of $G_2$), cf.~\cite[2.9]{MOV05}. 

Alternatively, one can in principle prove that the quasi-minuscule Schubert variety is singular by using Kumar's criterion as in \S\ref{sec_unitary} above. However, this is lengthy without computer aided calculations (cf.~also \cite[7.10]{MOV05} for a similar calculation in the split case), and would give less insight as we would not get the dimension of the tangent space at the identity.

\subsection{Construction of the twisted triality} Let $F'=k\rpot{u}$ be the cubic Galois extension of $F$ defined by $u^3=t$. The Galois group $I=\Gal(F'/F)$ is cyclic of order $3$, and we denote by $\tau\in I$ a generator. Then $\tau u=\zeta\cdot u$ where $\zeta=\zeta_3$ is a primitive third root of unity. 

The special orthogonal group in dimension $8$ is the functor on $k$-algebras $R$ given by
\begin{equation}\label{Special_Orthogonal}
\on{SO}_{8}(R)\,=\,\{A\in \SL_8(R)\,|\, AJA^tJ=\id\},
\end{equation}
where $J:=\on{antidiag}(1,\ldots, 1)\in \GL_8(R)$. Let $T'\subset B'\subset \on{SO}_{8}$ be the maximal diagonal torus, contained in the upper triangular Borel. The torus $T'$ is given by
\begin{equation}\label{Maximal_Torus}
T'(R)=\{\on{diag}(a_1,\ldots, a_4,a_4^{-1},\ldots, a_1^{-1})\,|\, a_1,\ldots,a_4\in R^\times\}.
\end{equation}
The coroot lattice $Q^\vee$ is the index $2$ sublattice of $X_*(T')=\bbZ^4$ with basis $\al_1^\vee=\epsilon_1-\epsilon_2$, $\al_2^\vee=\epsilon_2-\epsilon_3$, $\al_3^\vee=\epsilon_3-\epsilon_4$ and $\al_4^\vee=\epsilon_3+\epsilon_4$. Hence, $\pi_1(\on{SO}_8)=X_*(T')/Q^{\vee}= \bbZ/2$. 

Define $\pi\co H:=\on{Spin}_8\to \on{SO}_8$ to be the simply connected degree $2$ cover. The preimage $T_H:=\pi^{-1}(T')$ is a maximal torus of $H$ contained in the Borel subgroup $B_H:=\pi^{-1}(B')$. We have $X_*(T_H)=Q^\vee\subset X_*(T')$. The affine Dynkin diagram of $H$ is:
\begin{equation}\label{D_4}
\begin{aligned}
 \begin{tikzpicture}[scale=.4]
    \draw (0,0) node[anchor=east]  {$D_4$};
    \draw[xshift=1 cm,thin] (1 cm,0) circle (.1cm) node[above] {\tiny 0};
    \draw[xshift=2 cm,thin] (2 cm,0) circle (.1cm) node[above] {\tiny 2};
    \draw[xshift=4 cm,thin] (40: 17 mm) circle (.1cm) node[right] {\tiny 1};
    \draw[xshift=4 cm,thin] (0: 17 mm) circle (.1cm) node[right] {\tiny 3};
    \draw[xshift=4 cm,thin] (-40: 17 mm) circle (.1cm) node[right] {\tiny 4};
    \foreach \y in {1.15,...,1.15}
    \draw[xshift=\y cm,thin] (\y cm,0) -- +(1.4 cm,0);
    \draw[xshift=4 cm,thin] (40: 3 mm) -- (40: 14 mm);
     \draw[xshift=4 cm,thin] (0: 3 mm) -- (0: 14 mm);
    \draw[xshift=4 cm,thin] (-40: 3 mm) -- (-40: 14 mm);
  \end{tikzpicture}
\end{aligned}  
\end{equation}
We let $\sig_0\in \Aut(D_4)$ the automorphism defined by $2\mapsto 2$ and $1\mapsto 3\mapsto 4\mapsto 1$. We fix a principal nilpotent element in $X_H\in \Lie(B_H)$, and regard $\sig_0$ as an automorphism of $H$ via
\[
\Aut(D_4)=\Aut_{k}(H,B_H,T_H,X_H)\subset \Aut_{k}(H).
\]
Then the twisted triality $G$ is the functor on the category of $F$-algebras $R$ given by
\begin{equation}
G(R)\defined \{A\in H(R\otimes_FF')\,|\, \sig(A)=A\},
\end{equation}
where $\sig:=\sig_0\otimes \tau$. Likewise, we define the $F$-subgroups $T\subset B\subset G$ by descent. The affine Dynkin diagram of $G$ has the form
\begin{equation}\label{G_2}
\begin{aligned}
 \begin{tikzpicture}[scale=.4]
    \draw (0,0) node[anchor=east]  {$G_2^{I}$};
    \draw[xshift=1 cm,thin] (1 cm,0) circle (.1cm) node[above] {\tiny 0};
    \draw[xshift=2 cm,thin] (2 cm,0) circle (.1cm) node[above] {\tiny 2};
   \draw[xshift=3 cm,thin] (3 cm,0) circle (.1cm) node[above] {\tiny 1};
    \draw[xshift=1.15 cm,thin] (1.15 cm,0) -- +(1.4 cm,0);
    \draw[xshift=2.15 cm,thin] (2.15 cm,.08cm) -- +(1.4 cm,0);
      \draw[xshift=2.15 cm,thin] (2.15 cm,0) -- +(1.4 cm,0) node[left] {$<$};
    \draw[xshift=2.15 cm,thin] (2.15 cm,-.08cm) -- +(1.4 cm,0);
  \end{tikzpicture}
\end{aligned}  
\end{equation}
cf.~\cite[\S 4.2]{Ti77}. Here the arrow points to the shorter root, and should be regarded as an inequality sign saying that one root is of smaller length than the other. 

\subsection{The special parahoric}\label{Spec_Para} Since the extension $F'/F$ is tamely ramified, we have an identification of buildings $\scrB(H,F')^\sig=\scrB(G,F)$ compatible with the simplicial structure, cf.~\cite{PY02}. After conjugation by an element in $G_\ad(F)$, we reduce to the case where $x=0$ corresponds to the base point. In terms of parahoric group schemes, the pair $(H_{F'}, 0)$ corresponds to the $\calO_{F'}$-group $\calH:=H\otimes_k\calO_{F'}$. Hence, the special parahoric group scheme $\calG_x=\calG$ associated with the pair $(G,0)$ is the $\calO_F$-group $\calG\,=\, \Res_{\calO_{F'}/\calO_F}(\calH)^\sig$.

\subsection{Some loop groups}\label{Loop_Group} We denote by $LG$ (resp.~$L^+\calG$) the twisted loop group given on $k$-algebras $R$ by $LG(R)=G(R\rpot{t})$ (resp.~$L^+\calG(R)=\calG(R\pot{t})$). Likewise, we denote by $LH$ (resp.~$L^+\calH$) the loop group given by $LH(R)=H(R\rpot{u})$ (resp.~$L^+\calH(R)=\calH(R\pot{u})$). Then as $k$-group functors
\begin{equation}
LG\,=\, (LH)^\sig \;\;\;\;\;\; \text{(resp.~$L^+\calG=(L^+\calH)^\sig$)},
\end{equation}
which is an immediate consequence of the definition. The negative loop group $L^-\calH$ is defined on $k$-algebras $R$ by $L^-\calH(R)=H(R[u^{-1}])$. Let $L^{--}\calH:=\ker(L^-\calH\to H)$, $u^{-1}\mapsto 0$. Then the morphism given by multiplication
\begin{equation}\label{Mult_H}
L^{--}\calH\times L^+\calH\,\to\, LH, \;\; (h^-,h^+)\mapsto h^-\cdot h^+,
\end{equation}
is relatively representable by an open immersion, cf.~\cite[Prop 4.6]{LS97} (cf.~also \cite[Thm 2.3.1]{dHL}, \cite[Cor 3.2]{HRa}). The automorphism $\sig \in \Aut_k(LH)$ preserves the subgroup $L^{--}\calH\subset LH$, and we define the $k$-group
\[
L^{--}\calG \defined (L^{--}\calH)^\sig.
\]
By taking $\sig$-fixed points in \eqref{Mult_H}, we see that the multiplication morphism $L^{--}\calG\times L^+\calG\to\ LG$ is still an open immersion. 
Hence, if $e\in \Gr_G=LG/L^+\calG$ denotes the base point, then the morphism of $k$-ind-schemes
\begin{equation}\label{Open_Immersion}
L^{--}\calG\,\hookto\,\Gr_G, \;\; g^-\mapsto g^-\cdot e
\end{equation}
is representable by an open immersion.

\subsection{The quasi-minuscule Schubert variety}\label{Schubert}
Let $\breve{\Sig}$ be the \'echelonnage root system of $G$ which we give explicitly in \S\ref{Echennolage_Roots} below. Let $\bar{\mu}\in X_*(T)_I$ be the unique quasi-minuscule cocharacter for this root system. We fix an element $t^{\bar{\mu}}\in T(F)$ mapping to $\bar{\mu}$ under the Kottwitz morphism $T(F)\to X_*(T)_I$. We show in \S\ref{Quasi_Minuscule} that the element $t^{\bar{\mu}}$ maps under the map $T(F)\to T'(F')\subset \on{SO}_8(F')$ to the diagonal matrix
 \[
  \on{diag}(u^2,u,u,1,1,u^{-1},u^{-1},u^{-2})\cdot t_0,
 \]
for some $t_0\in T'(\calO_{F'})$. Let $C_{\bar{\mu}}\subset \Gr_G$ be the reduced $L^+\calG$-orbit of $t^{\bar{\mu}}\cdot e$. The {\it quasi-minuscule Schubert variety} $S_{\bar{\mu}}\subset \Gr_G$ is the closure of $C_{\bar{\mu}}$ equipped with the reduced scheme structure. Then $S_{\bar{\mu}}$ is a projective $k$-variety whose smooth locus contains $C_{\bar{\mu}}$. By the Cartan decomposition for twisted affine Grassmannians \cite[Cor 2.10]{Ri13}, we have
\begin{equation} \label{Cartan_decomp}
S_{\bar{\mu}}\,=\, C_{\bar{\mu}}\coprod \{e\}.
\end{equation}
\begin{lem} \label{RatSmooth_Lem}
The Schubert variety $S_{\bar{\mu}}$ is of dimension $6$. 
\end{lem}
\begin{proof} Let $2\rho =6\alpha_1 + 10\alpha_2 + 6\alpha_3 + 6\alpha_4$ be the sum of the positive roots in the absolute root system $\Phi_{D_4}$. By \cite[Cor.\,2.10]{Ri13}, 
$$
{\rm dim}(S_{\bar{\mu}}) = \langle 2\rho, \mu\rangle,
$$
where $\mu = 2\alpha_1^\vee + \alpha_2^\vee$ as in $\S\ref{Quasi_Minuscule}$ below. A calculation shows $\langle 2\rho, \mu  \rangle = 6$.
\end{proof}

\begin{rmk} \label{Ratl_Smooth_Rmk}
Note that $(G^\vee)^{I} = G_2$ by Lemma \ref{dual_list} iii). Under the geometric Satake isomorphism for ramified groups \cite{Zhu15, Ri16}, the Schubert variety $S_{\bar{\mu}}$ corresponds to the quasi-minuscule fundamental representation $V_{\bar{\mu}}$ of $G_2$. This is the unique $7$-dimensional non-trivial representation, it has $6$ extreme weights, and hence its trivial weight space is $1$-dimensional. This shows that $V_{\bar{\mu}}$ is weight-multiplicity-free, and hence $S_{\bar{\mu}}$ is rationally smooth  by Proposition \ref{ratl_sm_prop}, without using the classification result in Theorem \ref{mult_thm}.
\end{rmk}

\subsection{Various root systems}
We give explicitly the various root systems attached to the twisted triality. 

\subsubsection{$D_4$ roots}\label{D_4_Roots}
We use the notation of Bourbaki \cite{Bou} for the root system of type $D_4$.  The set of roots $\Phi_{D_4}$ carries the automorphism $\sigma_0$ of order $3$. The simple roots are 
$$
\alpha_1 = \epsilon_1 - \epsilon_2, \,\,\, \alpha_2 = \epsilon_2- \epsilon_3, \,\,\,  \alpha_3 = \epsilon_3-\epsilon_4, \,\,\, \alpha_4 = \epsilon_3 + \epsilon_4.
$$
We list the positive roots as $\sigma_0$-orbits:
\medskip
$$
\{\alpha_2\}, \,\, \{\alpha_1, \alpha_3, \alpha_4\}, \,\,\, \{\alpha_1+\alpha_2,\, \alpha_2+\alpha_3, \, \alpha_2 + \alpha_4\}
$$
$$
\{\alpha_1+ \alpha_2+ \alpha_3,\,\, \alpha_2+\alpha_3+ \alpha_4,\,\, \alpha_1+ \alpha_2+\alpha_4\} 
$$
$$
\{\alpha_1+\alpha_2+\alpha_3+\alpha_4\}, \,\,\,\, \{\alpha_1+2\alpha_2+\alpha_3+\alpha_4\}.
$$
The highest root is $\tilde{\alpha}^{D_4} = \alpha_1 + 2\alpha_2 + \alpha_3 + \alpha_4$.

\subsubsection{\'{E}chelonnage roots}\label{Echennolage_Roots}
The \'{e}chelonnage root system $\breve{\Sigma}$ for $G$ can be described explicitly in terms of the absolute roots $\Phi_{D_4}$, by \cite[Thm\,6.1]{Hai18}. The simple positive roots in $\breve{\Sigma}$ are the modified norms $N'_I(\alpha)$ of the the simple positive roots $\alpha$ for $\Phi_{D_4}$ (but here the modified norm coincides with the unmodified norm in \cite[Def.\,3.1]{Hai18}).  Therefore the simple positive \'{e}chelonnage roots may be written
$$
\breve{\Delta} = \{\alpha_2,\,\, \alpha_1 + \alpha_3 + \alpha_4 \}.
$$
In other words, $\alpha := \alpha_2$ is the short simple root, and $\beta := \alpha_1+ \alpha_3 + \alpha_4$ is the long simple root. It is evident from the angle $\angle(\alpha,\beta)$ that we get a root system of type $G_2$. Therefore in these coordinates the highest root is
$$
\tilde{\alpha} = 3\alpha + 2\beta = 3\alpha_2  + 2(\alpha_1 + \alpha_3+\alpha_4).
$$
The coroots are given by  
$$
\alpha^\vee = \alpha_2^\vee, \,\,\,\, \beta^\vee = \frac{\alpha_1^\vee + \alpha^\vee_3 + \alpha_4^\vee}{3}
$$
and (using that $\tilde{\alpha}^\vee =: \bar{\mu}$ is quasi-minuscule with respect to $\breve{\Sigma}$ hence is the fundamental coweight $\omega_\beta^\vee$)
\begin{equation} \label{coroot(highest)}
\tilde{\alpha}^\vee = \alpha^\vee + 2\beta^\vee = \alpha_2^\vee + \frac{2}{3}(\alpha_1^\vee+ \alpha_3^\vee + \alpha_4^\vee).
\end{equation}

\subsubsection{Quasi-minuscule coweight $\bar{\mu}$ for $G$}\label{Quasi_Minuscule}
The element (\ref{coroot(highest)}) is the result of applying the $\sigma$-averaging map $X_*(T)_I \rightarrow X_*(T)^I$ to an element $\mu = \alpha_2^\vee + 2 \alpha^\vee_1 \in X_*(T)$.  
By \cite[(7.3.2)]{Ko97}, under $T(F)/T(\calO_F) \rightarrow T(F')/T(\calO_{F'})$, $t^{\bar{\mu}}$ maps to $u^{\tilde{\mu}} := \tilde{\mu}(u) \in T(F')$, where 
\begin{equation} \label{tilde(mu)_def}
\tilde{\mu} := N\bar{\mu} = 2(\alpha_1^\vee + \alpha_3^\vee + \alpha_4^\vee) + 3\alpha_2^\vee  \in X_*(T).
\end{equation}
In terms of the diagonal torus $T'$ in ${\rm SO}_8$, the image of $t^{\bar{\mu}}$ takes the form
$$
u^{\tilde{\mu}} \cdot t_0 =  {\diag}(u^2, u, u, 1, 1, u^{-1}, u^{-1}, u^{-2})\cdot t_0,
$$
for some $t_0 \in T'(\mathcal O_{F'})$.

\subsubsection{Fixed-point roots} \label{fixed_rts_sec}
We can identify the root system of the fixed point group $H^{\sigma_0} = {\rm Spin}_8^{\sigma_0}$ using the general procedure of \cite[$\S4$]{Hai15}. The procedure is to take the non-divisible elements of the set of $\sigma_0$-averages of the roots of $H$.  We get the following list, corresponding to the $\sigma_0$-orbits listed above:
\medskip
$$
\{\alpha_2\}, \,\, \{\frac{\alpha_1 + \alpha_3 +\alpha_4}{3}\}, \,\,\, \{\alpha_2 + \frac{1}{3}(\alpha_1 + \alpha_3 + \alpha_4)\}
$$
$$
\{\alpha_2 + \frac{2}{3}(\alpha_1 + \alpha_3 + \alpha_4)\}
$$
$$
\{\alpha_1+\alpha_2+\alpha_3+\alpha_4\}, \,\,\,\, \{\alpha_1+2\alpha_2+\alpha_3+\alpha_4\}.
$$
This is the set of positive roots in a root system of type $G_2$ (cf.\,\cite[Pl.\,IX (II,V)]{Bou}). The group $H^{\sigma_0}$ is connected reductive, semi-simple, simple and simply connected, cf.\,Proposition \ref{fix_pt_prop}, so that $H^{\sigma_0}=G_2$ (we fix an isomorphism). Note that we have used the connectedness of $T^{\sigma_0}$, which is obvious in this situation.

\subsection{Nilpotent orbits with a twist}\label{Nil_Orbits}

Now assume $k=\bbC$.

\subsubsection{The space $(u^{-1}\mathfrak h)^\sigma$}
Write $\mathfrak h = {\rm Lie}(H)$ and $\mathfrak n = {\rm Lie}(H)^{\rm nilp}$, the set of nilpotent elements in $\mathfrak h$. Consider $u^{-1} \mathfrak h \subset \mathfrak h \otimes_{\mathbb C} \mathbb C\rpot{u}$.  This is a $\sigma$-stable $\mathbb C$-vector subspace of finite dimension. For a root $\gamma$ of $H$, let $u_\gamma: \mathbb C \rightarrow \mathfrak h$ be the corresponding Lie algebra homomorphism. We use the same symbol for $u_\gamma: \mathbb C\rpot{u} \rightarrow \mathfrak h \otimes_{\mathbb C} \mathbb C\rpot{u}$.

\begin{lem} \label{Orbits_Lem}
i\textup{)} We may identify $(u^{-1}\mathfrak h)^\sigma$ with the set of vectors
$$
\oplus_\gamma u_{\gamma}(u^{-1} x_\gamma) ~ \oplus ~ (y)
$$
where $\gamma$ ranges over the roots of $H$ and the $x_\gamma \in \mathbb C$ satisfy the condition that $x_{\sigma_0 \gamma} = \zeta^{-1}x_\gamma$ for all $\gamma$, and where $y \in (u^{-1}{\rm Lie}(T_H))^{\sigma}$.\smallskip\\
ii\textup{)} The vector space $(u^{-1}\mathfrak h)^{\sigma}$ is a 7-dimensional non-trivial representation of $H^\sigma = H^{\sigma_0}$, hence it is the quasi-minuscule fundamental representation of $G_2 = H^{\sigma_0}$. \smallskip\\
iii\textup{)} Fix $x \in \mathbb C\backslash \{0\}$. The variety $(u^{-1}\mathfrak n)^\sigma$ contains the reduced orbit closure $\overline{G_2 \cdot v_{\rm max}}$ of 
$$
v_{\rm max} := u_{\alpha_1 + \alpha_2 +\alpha_3}(u^{-1} x) \oplus u_{\alpha_2 + \alpha_3 + \alpha_4}(u^{-1}\zeta^{-1} x)\oplus u_{\alpha_1 + \alpha_2 + \alpha_4}(u^{-1}\zeta^{-2} x).
$$
This orbit closure is a $6$-dimensional affine $\bbC$-variety.
\end{lem}

\begin{proof}
Part (i) is immediate. We see that $x_\gamma = 0$ if $\gamma$ is $\sigma_0$-fixed. Therefore, we are left only with the contributions for $\gamma$ in
\begin{equation} \label{remaining_rts}
\pm \{ \alpha_1, \,\, \alpha_3, \,\, \alpha_4\}\, \cup \,\pm \{\alpha_1 + \alpha_2,\,\, \alpha_2 + \alpha_3,\,\, \alpha_2 + \alpha_4\}\, \cup\, \pm \{\alpha_1 + \alpha_2 + \alpha_3, \,\, \alpha_2 + \alpha_3 + \alpha_4, \,\, \alpha_1 + \alpha_2 + \alpha_4\}
\end{equation}
and the 1-dimensional space $ (u^{-1}{\rm Lie}(T_H))^{\sigma}$. Therefore ${\rm dim}_{\mathbb C}((u^{-1}\mathfrak h)^\sigma) = 7$.

Since $H$ and $\sigma$ act on $u^{-1}\mathfrak h$ such that $\sigma(h \cdot v) = \sigma(h) \cdot \sigma(v)$ for $h \in H$ and $v \in u^{-1}\mathfrak h$, we see that $H^\sigma$ acts on $(u^{-1}\mathfrak h)^\sigma$.  The action is visibly non-trivial. Since it is a 7-dimensional semisimple non-trivial representation, $(u^{-1}\mathfrak h)^\sigma$ must be the 7-dimensional representation associated to the quasi-minuscule fundamental weight of $G_2$. This proves (ii).

Finally, let $v_{\rm max}$ be as in (iii). It is a highest weight vector in the representation $(u^{-1}\mathfrak h)^\sigma$ of $H^\sigma = G_2$. Therefore its orbit is at least 6-dimensional because it contains the non-zero elements of the 6 extreme weight spaces in $(u^{-1}\mathfrak h)^\sigma$. Since $\mathfrak n$ is closed in $\mathfrak h$, the orbit closure is contained in $(u^{-1}\mathfrak n)^\sigma$.  As the latter space is $6$-dimensional (it does not contain $(u^{-1}{\rm Lie}(T_H))^{\sigma}$), we see that the orbit closure is exactly $6$-dimensional.

\end{proof}

\subsubsection{The reduced orbit closure is singular} We consider the reduced orbit closure $\overline{G_2 \cdot v_{\rm max}}$. The $G_2$-orbit is dense in a 6-dimensional vector subspace of $(u^{-1}\mathfrak h)^\sigma$, hence its closure contains the origin $0 \in (u^{-1}\mathfrak h)^\sigma$.

\begin{lem} \label{orbit_is_sing}
The 6-dimensional variety $\overline{G_2 \cdot v_{\rm max}}$ has a $7$-dimensional tangent space
$$
T_0(\overline{G_2 \cdot v_{\rm max}}) = (u^{-1}\mathfrak h)^\sigma,
$$
hence $\overline{G_2 \cdot v_{\rm max}}$ is singular at $0$.
\end{lem}
\begin{proof}
Clearly $T_0(\overline{G_2 \cdot v_{\rm max}}) \subseteq T_0((u^{-1}\mathfrak h)^{\sigma}) = (u^{-1}\mathfrak h)^\sigma$ as a $G_2$-invariant subspace.  Since $(u^{-1}\mathfrak h)^\sigma$ is an irreducible $G_2$-representation, the equality holds. The lemma follows from Lemma \ref{Orbits_Lem} ii). 
\end{proof}

\subsubsection{The exponential map} Essential to our proof of Theorem \ref{Sing_Thm} is the following proposition.
\begin{prop}\label{Exponential}
The exponential map 
\begin{equation}\label{Exp_Map}
\on{exp}\colon u^{-1}\frakn \to L^{--}\calH, \;\; u^{-1}X\mapsto \sum_{i=0}^\infty \frac{(u^{-1}X)^i}{i!}
\end{equation}
is algebraic and equivariant under $H$. Further it induces an algebraic map 
\[
\on{exp}\colon \overline{G_2 \cdot v_{\rm max}}\, \longrightarrow \, (L^{--}\calH)^{\sigma} \cap S_{\bar{\mu}} = L^{--}\calG \cap S_{\bar{\mu}}.
\]
\end{prop}

\begin{proof}
It is clear that \eqref{Exp_Map} is algebraic and $H$-equivariant. Both the scheme theoretic image $Z$ of $\on{exp}|_{\overline{G_2 \cdot v_{\rm max}}}$, and $L^{--}\calG \cap S_{\bar{\mu}}$ are reduced closed subschemes of $L^{--}\calH$. Thus, to show $Z\subset L^{--}\calG \cap S_{\bar{\mu}}$ we may argue on $\bbC$-points. In fact, for the remainder it suffices to show that $v_{\rm max}$ has $\sigma$-fixed image and lands in $C_{\bar{\mu}}$.

Let $U_\gamma : \mathbb C\rpot{u}^\times \to {\rm SO}_8 \otimes_{\mathbb C} \mathbb C\rpot{u}$ be the root group homomorphism associated to $\gamma$.
Then, by definition of $v_{\on{max}}$ in Lemma \ref{Orbits_Lem} iii), one has the formula
$$
\exp(v_{\rm max}) = U_{\alpha_1 + \alpha_2 + \alpha_3}(u^{-1}x) \cdot U_{\alpha_2 + \alpha_3 + \alpha_4}(u^{-1}\zeta^{-1}x) \cdot U_{\alpha_1 + \alpha_2 + \alpha_4}(u^{-1}\zeta^{-2} x).
$$
Note that the three root groups commute with each other, since no pair of the roots in $\{\alpha_1 + \alpha_2 + \alpha_3, \,\, \alpha_2 + \alpha_3 + \alpha_4, \,\, \alpha_1 + \alpha_2 + \alpha_4\}$ sum to a root of $H$. Thus $\exp(v_{\rm max})$ is evidently fixed by $\sigma$, so it lies in $L^{--}\calG$.  

It remains to show that $\exp(v_{\rm max})$ lies in $C_{\bar{\mu}}$. 
A calculation using the root groups shows that the image of $v_{\rm max}$ in ${\rm SO}_8(\mathbb C\rpot{u})$ is an $8 \times 8$ matrix of the form 
\begin{equation}\label{Matrix}
\left[
\begin{array}{cc  cc  cc  cc}
\hspace{.15in}1 \hspace{.15in}& &  &   u^{-1}x & u^{-1}\zeta^{-2}x&0 & 0& -u^{-2}\zeta^{-2} x^2 \\ [8pt]
&\hspace{.15in} 1 \hspace{.15in} &  & & & u^{-1}\zeta^{-1}x & 0 & 0\\ [8pt]
& & \hspace{.15in} 1 \hspace{.15in}\,\, & & &  & -u^{-1}\zeta^{-1}x & 0  \\ [8pt]
& & & \hspace{.15in} 1 \hspace{.15in} & & &  &  -u^{-1}\zeta^{-2}x \\ [8pt]
& & & & \hspace{.15in} 1 \hspace{.15in} & & & -u^{-1}x \\ [8pt]
& & & & & \hspace{.15in} 1 \hspace{.15in} & & \\ [8pt]
& & & & & & \hspace{.15in} 1 \hspace{.15in} &  \\ [8pt]
&& & & & & & \hspace{.15in} 1 \hspace{.15in}  
\end{array}
\right]
\end{equation}
where all unlabeled entries are $0$. 

Using the embedding ${\rm SO}_8 \subset {\rm GL}_8$ and $\S\ref{Quasi_Minuscule}$, it suffices to show that
$$
\exp(v_{\rm max}) \in K \,{\rm diag}(u^{-2}, u^{-1}, u^{-1}, 1, 1, u, u , u^2)\, K
$$
where $K = {\rm GL}_8(\mathbb C\pot{u})$. We prove this by using the algorithm for finding the Smith Form of a matrix over a PID. For the matrix $A = \exp(v_{\rm max})$ and $1 \leq i \leq 8$, define the weakly-increasing sequence of integers $a_i \in \bbZ$ by requiring that
$$
(u^{\sum_{j \leq i} a_j}) = \mbox{ideal generated by the $i \times i$ minors of $A$}.
$$
Then the Smith Form of $A$ is the matrix ${\rm diag}(u^{a_1}, u^{a_2}, \dots, u^{a_8})$. An inspection of \eqref{Matrix} shows that $a_1 = -2$, $a_1 + a_2 = -3$, $a_1 + a_2 + a_3 = -4$, and $a_1 + a_2 + a_3 + a_4 = -4$.  Thus the first half of the entries of the Smith Form of $\exp(v_{\rm max})$ is $(u^{-2}, u^{-1}, u^{-1}, 1)$.  Because the matrix is in ${\rm SO}_8$, this determines the rest of the entries as well. This completes the proof.
\end{proof}

\subsection{End of the proof}\label{End_Proof}

We now finish the proof of Theorem \ref{Sing_Thm}. In Proposition \ref{Exponential}, we have constructed an algebraic morphism of $6$-dimensional $\bbC$-varieties
\begin{equation}\label{Exponential_2}
\exp \colon \overline{G_2 \cdot v_{\rm max}} \, \longrightarrow \, L^{--}\calG \cap S_{\bar{\mu}},
\end{equation}
cf.~Lemmas \ref{RatSmooth_Lem} and \ref{Orbits_Lem} iii) for the dimension. The target is an open neighborhood of the base point $e$ in $S_{\bar{\mu}}$. Note that $0\mapsto e$ under \eqref{Exponential_2}. By Lemma \ref{orbit_is_sing}, it suffices to show that \eqref{Exponential_2} is an open immersion, i.e., an isomorphism onto an open neighborhood of $e$. For this, we argue as follows. 

The map $\exp\co u^{-1}\frakh\to L^{--}\calH$ is injective, as can be seen by writing down the formula for the exponential map in $\on{SO}_8$, and by comparing the coefficients of $u^{-1}$. Hence, the map \eqref{Exponential_2} is an injective morphism of irreducible affine $\bbC$-varieties of the same dimension, and in particular birational (because the field extension at the generic points is separable, so that it must have degree $1$). As the target is normal by \cite[Thm 6.1]{PR08}, the map \eqref{Exponential_2} is an open immersion by Zariski's main theorem \cite[Cor 4.4.9]{EGA3}. This completes the proof of Theorem \ref{Sing_Thm}.

\appendix
\section{A remark on fixed point groups}
Let $H$ be a connected reductive group over an algebraically closed field $\kappa$, and let $(T,B, X)$ be a pinning which is preserved by the action on $H$ of a finite group $I$. Recall that $H^I$ is a (possibly disconnected) reductive group, with maximal torus the neutral component $T^{I,\circ}$ of the diagonalizable subgroup $T^I \subset H^I$ (cf.\,\cite[Prop.\,4.1]{Hai15}). In what follows, $Z(A)$ denotes the center of an algebraic group $A$.

\begin{prop} \label{fix_pt_prop} Assume ${\rm char}(\kappa) \neq 2$. Then the following statements hold:\smallskip\\
i\textup{)} If $T^I$ is connected \textup{(}e.g.,\,$H$ is adjoint or simply-connected\textup{)}, then $H^I$ is connected.\smallskip\\
ii\textup{)} $Z(H^I) = Z(H)^I$, and this group contains $Z(H^{I, \circ})$.\smallskip\\
iii\textup{)} If $H$ is semi-simple, then $H^I$ is semi-simple.\smallskip\\
iv\textup{)} If $H$ is adjoint, then $H^I$ is adjoint.\smallskip\\
v\textup{)} If $H$ is simple and simply-connected, then $H^I$ is simple, and simply-connected unless $H$ is of type $A_{2n}$ and carries a non-trivial $I$ action, in which case $H \cong {\rm SL}_{2n+1}$ and $H^I \cong {\rm SO}_{2n+1}$.
\end{prop}

\begin{proof}
By e.g.\,\cite[Prop.\,4.1]{Hai15} we know $\pi_0(T^I) \overset{\sim}{\rightarrow} \pi_0(H^I)$. If $H$ is adjoint or simply-connected, then $X^*(T)$ is an induced $I$-module, and hence $T^I$ is connected because $X^*(T^I) = X^*(T)_I$ is $\bbZ$-free. This proves (i). Note this part holds with no assumption on ${\rm char}(\kappa)$.

For (ii), note that $(T_{\rm ad})^I$ is connected, and thus by \cite[Prop.\,4.6]{Hai15} we have $Z(H)^I ~ H^{I, \circ} = H^I$, and hence $Z(H)^I ~ Z(H^{I,\circ}) = Z(H^I)$. Now by \cite[proof of Prop.\,4.1]{Hai15}, when ${\rm char}(\kappa) \neq 2$ the simple roots for $(H^{I,\circ}, T^{I,\circ})$ consist precisely of the restrictions to $T^{I, \circ}$ of the simple roots for $(H,T)$.  It follows that $z \in Z(H^{I, \circ})$ is killed by all the roots of $(H,T)$, hence belongs to $Z(H)$. This proves $Z(H^{I,\circ}) \subseteq Z(H)^I$, and (ii) follows.  Parts (iii) and (iv) follow from (ii).

Part (v): Assume $H$ is simple and simply connected. By (i), $H^I$ is connected. Note that if $\tilde{\alpha}$ is the highest root for $(H, T)$, then the $I$-average $\tilde{\alpha}^\diamond$ is highest for $(H^{I,\circ}, T^{I,\circ})$, and hence the root system for the latter is irreducible since it has a unique highest root. This proves $H^I$ is simple.  Consider the dual group $H^{I, \vee}$, with dual torus $T^{I, \vee}$. It is enough to determine when $H^{I, \vee}$ is adoint. We have $X_*(T^{I, \vee}) = X^*(T^I) = X_*(T^\vee)_I$ is dual to $X^*(T^{\vee})^I = {\rm Hom}(X_*(T^\vee)_I, \bbZ)$ (note that $X_*(T^\vee)_I$ is free).  Therefore $X^*(T^{I, \vee}) \cong X^*(T^\vee)^I = (\bbZ\Phi^\vee(H))^I$.  Therefore to show $H^{I, \vee}$ is adjoint, it is equivalent to show that
$$
\bbZ \Phi^\vee(H^I) ~ \overset{\sim}{\hookrightarrow} ~ (\bbZ\Phi^\vee(H))^I.
$$
The right hand side has as $\bbZ$-basis the set $N_I(\Delta^\vee(H))$ of unmodified norms of simple coroots of $H$  (see the definition of the operations $N_I$ and $N'_I$ in \cite[Def.\,3.1]{Hai18}\footnote{We take this opportunity to point out a typo in the definition of $v^\diamond$ above \cite[Def.\,3.1]{Hai18}: it should read $v^\diamond := \frac{1}{|I|} \sum_{\sigma \in I} \sigma(v)$.}). On the other hand, using the notation of \cite[$\S3$]{Hai18}, the proof of \cite[Prop.\,4.1]{Hai15} shows that $\Phi(H^I) = {\rm res}_I\Phi(H) \cong (\Phi(H)^\diamond)_{\rm red}$. By the duality result of \cite[Prop.\,3.5]{Hai18}, $\bbZ\Phi^\vee(H^I)$ has a $\bbZ$-basis given by $N'_I(\Delta^\vee(H))$.  Therefore, $H^{I, \vee}$ is adjoint if and only if $N_I'(\Delta^\vee(H)) = N_I(\Delta^\vee(H))$, which happens if and only if the $I$-action is trivial or $\Phi(H)$ is not of type $A_{2n}$.  The rest of the assertions of part (v) are clear.
\end{proof}


\begin{thebibliography}{99999999}










\bibitem[Arz09]{Arz09} K.\,Arzdorf: {\it On local models with special parahoric level structure}, Michigan Math.\,J.\,{\bf 58} (2009), no.3, \,683-710.





\bibitem[BBD82]{BBD82} A.\, Beilinson, J.\, Bernstein, P.\, Deligne: {\it Faisceaux pervers}, Ast\'erisque \textbf{100} (1982), 5-171.



\bibitem[BZ90]{BZ90} A. D. Berenshtein, A. V. Zelevinskii: {\it When is the multiplicity of a weight equal to $1$?}, Funct. Anal. Appl. \textbf{24} (1990), no. 4, 259-269. 



\bibitem[BL00]{BL00} S.\,Billey, V.\,Lakshmibai: {\it Singular Loci of Schubert Varieties}. Progress in Math.\,{\bf 182}, Birkh\"{a}user, 2000. 251 pp.\,+xii.



\bibitem[Bou]{Bou} N. Bourbaki: {\it \'El\'ements de Math\'ematique. Fasc. XXXIV. Groupes et alg\`ebre de Lie. Chapitre: IV: Groupes de Coxeter et syst\`eme de Tits. Chapitre V: Groupes engendr\'es par des r\'eflexions. Chapitre VI: Syst\`emes de racines}, in: Actualit\'es Scientifiques et Industrielles, vol. 1337, Hermann, Paris, 1968, 288 pp.





\bibitem[BT72]{BT72} F. Bruhat and J. Tits: {\it Groupes r\'eductifs sur un corps local I. Donn\'ees radicielles valu\'ees}, Inst. Hautes \'Etudes Sci. Publ. Math. \textbf{41} (1972), 5-251. 

\bibitem[BT84]{BT84} F.\, Bruhat, J.\, Tits: {\it Groupes r\'eductifs sur un corps local II. Sch\'ema en groupes. Existence d'une donn\'ee radicielle valu\'ee}, Inst. Hautes \'Etudes Sci. Publ. Math. \textbf{60} (1984), 197-376. 










\bibitem[dHL]{dHL} M.\,A.\,de Cataldo, T.\,Haines, and L.\,Li: {\it Frobenius semisimplicity for convolution morphisms}, Math.\,Zeitschrift \,{\bf 289} (2018), 119-169. DOI 10.1007/s00209-017-1946-4.













\bibitem[EM99]{EM99} S. Evens and I. Mirkovi\'c, {\it Characteristic cycles for the loop Grassmannian and nilpotent orbits}, Duke Math. J. \textbf{97} (1999), no. 1, 109-126.

\bibitem[EGA3]{EGA3} A.\,Grothendieck, J.\,Dieudonn\'e: {\it El\'ements de g\'eom\'etrie alg\'ebrique: III. \'Etude cohomologique des faisceaux coh\'erents, Premi\`ere partie}, Publications Math\'ematiques de l'IH\'ES \textbf{11}, pp. 5-167.
















  






\bibitem[Hai14]{Hai14} T.\,Haines: {\it The stable Bernstein center and test functions for Shimura varieties}, In {\it Automorphic Forms and Galois Representations} vol.2, edited by F.~Diamond, P.~Kassaei and M.~Kim, 118-186.  London Math.~Soc.~Lecture Notes {\bf 415}. Cambridge University Press, 2014.



\bibitem[Hai15]{Hai15} T.\,Haines: {\it On Satake parameters for representations with parahoric fixed vectors}, IMRN \textbf{20} (2015), 10367-10398.


\bibitem[Hai18]{Hai18} T.\,Haines: {\it Dualities for root systems with automorphisms and applications to non-split groups},  Representation Theory {\bf 22} (2018), 1-26.




\bibitem[HLR]{HLR} T.\,Haines, J.~N.~P.~Louren\c{c}o, T.\,Richarz: {\it On the normality of Schubert varieties in twisted affine flag varieties}, new version of arXiv:1806.11001, in preparation.

\bibitem[HRa]{HRa} T.\,Haines, T.\,Richarz: {\it The test function conjecture for parahoric local models}, preprint (2018), arXiv:1801.07094. 

\bibitem[HRb]{HRb} T.\,Haines, T.\,Richarz: {\it The test function conjecture for local models of Weil restricted groups}, preprint (2018), 	arXiv:1805.07081. 

\bibitem[HRc]{HRc} T.\,Haines, T.\,Richarz: {\it Normality and Cohen-Macaulayness of parahoric local models}, preprint. arXiv:1903.10585.




\bibitem[HPR]{HPR} X.\,He, G.\,Pappas, M.\,Rapoport: {\it Good and semi-stable reductions of Shimura varieties}, preprint, arXiv:1804.09615.





\bibitem[Ho95]{Ho95} R. Howe: {\it Perspectives on Invariant Theory: Schur duality, multiplicity-free actions and beyond}, Israel Conf. Math. Proc. \textbf{8} (1995), 1-182. 






\bibitem[KL79]{KL79} D.\,Kazhdan, G.\,Lusztig: {\it Representations of Coxeter groups and Hecke algebras}, Invent. Math. {\bf 53} (1979), 165-184.

\bibitem[KL80]{KL80} D.\,Kazhdan, G.\,Lusztig: {\it Schubert varieties and 
  Poincar\'{e} duality}, Proc.\,Symp.\,Pure Math.\,{\bf 36} (1980), 185--203.






\bibitem[KW01]{KW01} R.\,Kiehl and R.\,Weissauer: {\it Weil conjectures, perverse sheaves and $\ell$-adic Fourier transform}, Springer (2001), pp 375.





\bibitem[Ko97]{Ko97} R.\,Kottwitz: {\it Isocrystals with additional structures II}, Compos. Math. \textbf{109} (1997), 255-339.







\bibitem[Ku96]{Ku96} S.\,Kumar: {\it The nil Hecke ring and singularity of Schubert varieties}, Invent.\,Math.\,{\bf 123} (1996), no.\,3, 471-506.








\bibitem[LS97]{LS97} Y.\,Laszlo, C.\,Sorger: {\it The line bundles on the moduli of parabolic $G$-bundles over curves and their sections}, Ann. Sci. \'Ecole Norm. Sup. (4) \textbf{30} (1997), no. 4, 499-525.









\bibitem[Lu83]{Lu83} G.\,Lusztig: {\it Singularities, character formulas, and a $q$-analogue of weight multiplicities}, Analysis and Topology on Singular Spaces, II, III, Ast\'erisque, Luminy, 1981, vol. \textbf{101-102}, Soc. Math. France, Paris (1983), pp. 208-229.






\bibitem [Lus03]{Lus03} G.\,Lusztig, {\em Hecke algebras with unequal parameters}, CRM Monographs Ser. {\bf 18}, Amer.\,Math.\,Soc.\, 2003, 136p.


\bibitem[MOV05]{MOV05} A. Malkin, V. Ostrik and M. Vybornov, {\it The minimal degeneration singularities in the affine Grassmannians}, Duke Math. J. \textbf{126} (2005), 233-249.












\bibitem[Mue08]{Mue08} A.\,M\"uller: {\it Singularit\"aten minimaler Degenerationen in affinen Grassmannschen, insbes.~in positiver Charakteristik}, Diplomarbeit (2008), Bonn University.

\bibitem[NP01]{NP01} Ng\^o B.\,C.\,and P.\,Polo: {\it R\'esolutions de Demazure affines et formule de Casselman-Shalika g\'eom\'etrique}, J.\,Algebraic Geom.\,\textbf{10} (2001), no. 3, 515-547.

\bibitem[P00]{P00} G.\,Pappas: {\it On the arithmetic moduli schemes of PEL Shimura varieties}, J. Alg. Geom. \textbf{9} (2000), 577-605.


\bibitem[PR08]{PR08} G.\,Pappas, M.\,Rapoport: {\it Twisted loop groups and their affine flag varieties}, Adv. Math. \textbf{219} (2008), 118-198. 

\bibitem[PR09]{PR09} G.\,Pappas, M.\,Rapoport: {\it Local models in the ramfied case. III. Unitary groups}, J.\,Inst.\,Math.\,Jussieu {\bf 8} (2009), 507-564.


\bibitem[PZ13]{PZ13} G.\,Pappas, X.\,Zhu: {\it Local models of Shimura varieties and a conjecture of Kottwitz}, Invent.\,Math.\,\textbf{194} (2013), 147-254.

\bibitem[PY02]{PY02} G.\,Prasad, J.-K.\,Yu: {\it On finite group actions on reductive groups and buildings}, Invent. Math. \textbf{147} (2002), 545-560.











\bibitem[Ri13]{Ri13} T. Richarz: {\it Schubert varieties in twisted affine flag varieties and local models}, Journal of Algebra \textbf{375} (2013), 121-147.




\bibitem[Ri16]{Ri16} T.\,Richarz: {\it Affine Grassmannians and geometric Satake equivalences}, IMRN \textbf{12} (2016), 3717-3767.  















\bibitem[StaPro]{StaPro} Stacks Project, Authors of the stacks project, available at http://stacks.math.columbia.edu/. 






\bibitem[Ti77]{Ti77} J.\,Tits: {\it Reductive groups over local fields}, Automorphic forms, representations and $L$-functions, in: Proc. Sympos. Pure Math., Corvallis, OR, 1977, vol. XXXIII, Amer. Math. Soc., Providence, RI, 1979, pp. 29-69. 







\bibitem[Zhu15]{Zhu15} X.\,Zhu: {\it The Geometrical Satake Correspondence for Ramified Groups}, with an appendix by T. Richarz and X. Zhu, Ann. Sc. de la ENS, s\'erie \textbf{48}, fascicule 2 (2015), 409-451.


\end{thebibliography}
\end{document}